\documentclass[11pt]{article}

\usepackage{amsmath}\usepackage{amsfonts}\usepackage{amssymb}\usepackage{amsthm}
\usepackage{graphicx}
\usepackage{ifthen}
\usepackage{lscape}
\usepackage{subfigure}
\usepackage{setspace}
\usepackage{subeqnarray}

\usepackage[left=1in,top=1.5in,right=1in,bottom=1in]{geometry}

\usepackage{verbatim}
\usepackage{mdwlist}
\usepackage{xcolor}
\usepackage{hyperref}

\newtheorem{theorem}{Theorem}
\newtheorem{remark}{Remark}
\newtheorem{corollary}{Corollary}
\newtheorem{lemma}{Lemma}
\newtheorem{proposition}{Proposition}

\long\def\symbolfootnote[#1]#2{\begingroup\def\thefootnote{\fnsymbol{footnote}}
\footnote[#1]{#2}\endgroup}

\date{}
\author{Assyr Abdulle\thanks{Section of Mathematics, Swiss Federal Institute of Technology (EPFL), Station 8, CH-1015, Lausanne, Switzerland
}
\and Ping Lin\thanks{Division of Mathematics, University of Dundee, 23 Perth Road, Dundee, Scotland DD1 4HN, UK%
}
\and Alexander V.\ Shapeev\footnotemark[1]
}

\title{Homogenization-based Analysis of Quasicontinuum Method for Complex Crystals}

\renewcommand{\vec}{\boldsymbol}
\newcommand{\vecu}{\vec{u}}

\newcommand{\baru}{\bar{u}}
\newcommand{\tildeu}{\tilde{u}}
\newcommand{\eps}{\epsilon}

\newcommand{\per}{{\rm per}}

\newcommand{\Eint}{E_{\rm int}}
\newcommand{\Eext}{E_{\rm ext}}

\newcommand{\calR}{{\cal R}}
\newcommand{\bbR}{{\mathbb R}}

\begin{document}
\sloppy

\maketitle

\begin{abstract}
Among the efficient numerical methods based on atomistic models, the quasicontinuum (QC) method, introduced by Tadmor, Ortiz, and Phillips (1996), has attracted growing interest in recent years.
Originally, the QC method was developed for materials with simple crystalline lattice (simple crystals) and later was extended to complex lattice (Tadmor et al, 1999).
In the present paper we formulate the QC method for complex lattices in a homogenization framework and perform analysis of such a method in a 1D setting.
We also present numerical examples showing that the convergence results are valid in a more general setting.
\end{abstract}

\section{Introduction}

In some applications of solid mechanics, such as modeling cracks, structural defects, or nanoelectromechanical systems (NEMS), the classical continuum description is not suitable, and it is required to utilize an atomistic description of materials.
However, full atomistic simulations are prohibitively expensive, hence there is a need for efficient numerical method.
Among the efficient methods based on atomistic models, the quasicontinuum (QC) method has attracted growing interest in recent years \cite{MillerTadmor2002}.

The QC method is a multiscale method capable of coupling atomistic and continuum description of materials.
It is intended to model an atomistic material in a continuum manner in the regions where deformation variations are low and use fully atomistic model only in the small neighborhood of defects, thus effectively reducing the degrees of freedom of the system.
Originally, the QC method was developed for materials with simple crystalline lattice \cite{TadmorPhillipsOrtiz1996} and the convergence of a few variants of the method has been analyzed under some practical assumptions (see, e.g., \cite{DobsonLuskin2008, MingYang2009, Lin2007, OrtnerSuli2008}).
The QC method is based on the so-called Cauchy-Born rule (see, e.g., \cite{Ericksen2008, FrieseckeTheil2002, EMing2007, BlancLeBrisLions2007a}) which states that the energy of a certain volume of material can be approximated through the deformation energy density, which is computed for a representative atom assuming that the neighboring atoms follow a uniform deformation.
Later, QC was extended to materials with complex lattice (a union of a number of simple lattice sites) \cite{TadmorSmithBernsteinEtAl1999} based on the improved Cauchy-Born rule \cite{Stakgold1950} which accounts for relative shifts between the comprising simple lattice sites.
Examples of such materials include diamond cubic \mbox{Si}, HCP metals (stacking two simple hexagonal lattices with a shift vector) like \mbox{Zr}, ferroelectric materials, salts like Sodium Chloride, and intermetallics like \mbox{NiAl}.
Recent developments of QC for complex lattices also include adaptive choice of representative cell of complex crystals \cite{DobsonElliottLuskinEtAl2007}.
It appears that no rigorous analysis is available so far for the complex lattice QC.

In the present work we propose a treatment of complex crystalline materials within the framework of discrete numerical homogenization. Homogenization techniques for partial differential equations (PDEs) with multiscale coefficients
are known to be successful for obtaining effective equations with coefficients properly averaged out
\cite{BensoussanLionsPapanicolaou1978}. Finite element methods based on homogenization theory have been pioneered by
Bab\u{u}ska \cite{Babuska1976} and have attracted growing attention these past few years
(see \cite{GeersKouznetsovaBrekelmans2010, Abdulle2009, EEL2007, EfendievHou2009} for textbooks or review papers).
Following the ideas of \cite{BensoussanLionsPapanicolaou1978}, we use homogenization techniques to describe
the coarse-graining of complex lattice. 
This allows to give a new formulation of the QC method for complex lattice (that we will sometimes call ``homogenized QC" (HQC) method).
More interestingly, we find that there is equivalence between the discretely homogenized QC and the complex lattice QC based on the improved Cauchy-Born rule.
We can then use this discrete homogenization as a framework for the description of a quasicontinuum method for complex materials. There are several benefits in this regard.
First, in this framework the connection to the well-developed theory of continuum homogenization and related numerical methods
becomes more apparent.
This allows us to apply the analysis techniques developed for continuum homogenization \cite{Abdulle2009, EEL2007} to the quasicontinuum method for complex materials.
Second, homogenization theory can be used to upscale the atomistic model in both, time and space, which makes it promising for modeling and especially analyzing zero temperature and finite temperature motion of atomistic materials \cite{EEL2007, FishChenLi2007,MillerTadmor2002}.
Also, homogenization can be applied to ``stochastic'' materials, atomistic counterparts of which include polymers \cite{BaumanOdenPrudhomme2009} and glasses.
Last, for the finite temperature simulations, when materials are modeled with static atoms interacting with effective temperature-dependent potentials, homogenization may serve as a rigorous instrument to derive such potentials.

We note that the idea of applying homogenization to atomistic media has appeared in the literature \cite{Chung2004, ChungNamburu2003, ChenFish2006, FishChenLi2007, BaumanOdenPrudhomme2009}.
We also note that the method considered in this paper is essentially equivalent to the QC for complex crystals,
being put in the framework of numerical homogenization.\footnote{%
For more details on relations of different multiscale methods
for complex crystalline materials, refer to the companion paper
\cite{AbdulleLinShapeevII}.}
However, the rigorous discrete homogenization procedure and related numerical method allow
us to derive error estimates for the homogenized QC method, when compared to the
solution of discretely homogenized atomistic equations.
It also allows, by a reconstruction procedure, to approximate the original full atomistic solution.
To the best of our knowledge,
such error estimates are new.
As in many numerical homogenization
techniques for PDEs, there is no need for our numerical approximation to derive
homogenized potential before-hand, since the effective potential is computed
on the fly (see, e.g., \cite{EEL2007}). Finally, we note that the error estimates are derived in one dimension for linear interaction, but
the numerical methods itself applies to nonlinear multi-dimensional problems. Numerical
experiments show that the derived estimates are valid in more general situations.

The paper is organized as follows.
After a brief presentation of a 1D model problem of atomistic equilibrium (Section \ref{sec:problem_formulation}) we discuss discrete homogenization (Sections \ref{sec:atm-hmg} and \ref{sec:estimates}), and then formulate and analyze a
macro-micro numerical method
capable of capturing effective behavior of a complex material (Sections \ref{sec:HQC} and \ref{sec:HQC-convergence}).
We also illustrate how the presented technique can be applied to 2D crystals (Section \ref{sec:2d}).
Numerical examples illustrating the performance of our method are then presented (Section \ref{sec:numeric}), followed by concluding remarks (Section \ref{sec:conclusion}).

\subsection{Function spaces}
\label{sub:fct_space}

We consider the space of $n$-periodic functions on the lattice $\delta\mathbb Z$ ($\delta\in\bbR$, $\delta>0$):
\[
U_\per^n(\delta\mathbb{Z})= \left\{u: {\delta\mathbb Z}\to\bbR: \ u(X_i) = u(X_{i+n})\ \forall i\in{\mathbb Z} \right\}.
\]
and the space of $n$-periodic sequences with zero average:
\[
U_{\#}^n(\delta\mathbb{Z})= \left\{u\in U_\per^n(\delta\mathbb{Z}):~\left<u\right>_X=0\right\},
\]
where the discrete integration operator $\left<\bullet\right>_X$ is defined for $u\in U_\per^n(\delta\mathbb{Z})$
by
\[
\left<u\right>_X= \frac{1}{n}\sum\limits_{i=1}^{n} u(X_i).
\]
Likewise, we consider the tensor product space on $\delta_1\mathbb Z\times \delta_2\mathbb Z$:
\begin{align*}
U_\per^{n_1}(\delta_1\mathbb{Z})\otimes U_\per^{n_2}(\delta_2\mathbb{Z})
= \{ ~& u:\delta_1\mathbb Z\times \delta_2\mathbb Z\to \bbR:
\\ ~&
u(X_i,\bullet) = u(X_{i+n_1},\bullet),~~ u(\bullet, Y_j) = u(\bullet,Y_{i+n_2})\quad \forall i,j\in{\mathbb Z}\},
\end{align*}
and discrete integration operators
\[
\left<u\right>_X= \frac{1}{n_1}\sum\limits_{i=1}^{n_1} u(X_i,Y_j),
\quad
\left<u\right>_Y= \frac{1}{n_2}\sum\limits_{j=1}^{n_2} u(X_i,Y_j),
\quad
\left<u\right>_{XY}= \left<\left<u\right>_Y\right>_X= \left<\left<v\right>_X\right>_Y.
\]
A bilinear form for $u,v\in U_\per^n(\delta\mathbb{Z})$ is defined by
\[
\left<u,{v}\right>_X = \frac{1}{n}\sum\limits_{i=1}^{n} u(X_i)\, v(X_i),
\]
and for $u,v\in U_\per^{n_1}(\delta_1\mathbb{Z})\otimes U_\per^{n_2}(\delta_2\mathbb{Z})$ by
\[
\left<u,{v}\right>_{XY} = \frac{1}{n_1 n_2}\sum\limits_{i=1}^{n_1} \sum\limits_{j=1}^{n_2} u(X_i,Y_j)\, v(X_i,Y_j),
\]
For $u\in U_\per^n(\delta\mathbb{Z})$ we introduce the forward discrete derivative $D u\in U_\per^n(\delta\mathbb{Z})$
\[
D u(X_i) = \frac{u(X_{i+1})-u(X_i)}{X_{i+1}-X_i}= \frac{u(X_{i+1})-u(X_i)}{\delta},
\]
and the $r$-step discrete derivative ($r\in\mathbb Z,r\neq 0$)
$D_{r}{v}\in U_\per^n(\delta\mathbb{Z})$
\[
D_{r} u(X_i) =\frac{u(X_{i+r})-u(X_i)}{X_{i+r}-X_i}= \frac{u(X_{i+r})-u(X_i)}{r\delta}.
\]
In addition to differentiation operators, we define for $u\in U_\per^n(\delta\mathbb{Z})$, the translation operator
$Tu\in U_\per^n(\delta\mathbb{Z})$
\[
Tu(X_i) = u(X_{i+1}).
\]
Then the $r$-step translation ($r\in\mathbb Z$) can be expressed as a power of $T$:
\[
T^r u(X_i) = u(X_{i+r}).
\]
The definitions of the discrete derivative and translation generalize to functions in
$U_\per^{n_1}(\delta_1\mathbb{Z})\otimes U_\per^{n_2}(\delta_n\mathbb{Z})$
by considering the partial discrete derivative and translation operators, i.e.,
$D_X,D_{X,r},T_{X}$ applied to $u(\bullet,Y_j)$ and
$D_Y,D_{Y,r},T_{Y}$ applied to $u(X_i,\bullet)$.

The following lemma, whose proof is trivial, will be useful:
\begin{lemma}[Discrete integration by parts]\label{lem:integration-by-parts}
For $u,{v}\in U_\per^n(\delta\mathbb{Z})$ the following identity holds:
\[
\left<u, D_r {v} \right>_X
=
- \left< T^{-r} D_{r} u, {v} \right>_X.
\]
This identity can be written in an operator form as $(D_r)^* = T^{-r} D_{r}$.
\end{lemma}

We finally define appropriate norms for functions $v\in U_\per^n(\delta\mathbb{Z})$:
\begin{equation}
\begin{array}{c} \displaystyle
\|{v}\|_{L^q(n)} = \left(\frac{1}{n}\sum\limits_{i=1}^{n} |v(X_i)|^q\right)^{1/q},
\\ \displaystyle
\|{v}\|_{L^{\infty}(n)} = \max\limits_{1\le i\le n} |v(X_i)|,
\quad
|{v}|_{W^{1,q}(n)} = \|D {v}\|_{L^q(n)},
\\ \displaystyle
|{v}|_{H^1(n)} = |{v}|_{W^{1,2}(n)},
\quad
|{v}|_{H^2(n)} = \|D^2 {v}\|_{L^2(n)},
\quad
|{v}|_{H^{-1}(n)} = \sup\limits_{\substack{w\in U_{\#}^n(\delta\mathbb{Z}) \\ w\ne 0}} \frac{\left<v,w\right>_X}{~~|{w}|_{H^1(n)}}.
\end{array}
\label{eq:estimates:norms_fct}
\end{equation}

\subsubsection{Identification in $\bbR_\per^n$}
\label{rk:u_identification}
It is clear that a function $u\in U_\per^n(\delta\mathbb{Z})$ can be identified with a representant $\vecu =\left[u_i\right]_{i=1}^{n}$
in $\bbR_\per^n$, where $u_i=u(X_i)$ (the subscript $_\per$ means that $u_i$ is defined by periodic extension $u_{i+n}=u_i$ for all indices $i\in\mathbb Z$).
We can also identify functions in $U_{\#}^n(\delta\mathbb{Z})$ with their representants in $\bbR_\per^n$
with zero mean.
We will denote this vector space as $\bbR_{\#}^n$.
In this paper we will use a product space of $U_\per^n(\delta\mathbb{Z})$ with different values of $\delta$.
In such a case it is important to retain the functional notation for $u$.
However, when there is no confusion, we will avoid such heavy notations and simply use $\vecu ,D_r\vecu ,T\vecu \in\bbR_\per^n$
where due to identification of $u_i$ with $u(X_i)$, the operators are defined as
\begin{eqnarray}
(D_r\vecu )_i=D_r u_i=\frac{u_{i+r}-u_i}{r\delta},
\quad
(T\vecu )_i=Tu_i=u_{i+1},
\end{eqnarray}
$D_1$ will simply be denoted as $D$.
Likewise the discrete integration and bilinear form can be written as
\[
\left<\vecu \right>_i = \frac{1}{n}\sum\limits_{i=1}^{n} u_i,
\quad
\left<\vecu ,\vec{v}\right>_i = \frac{1}{n}\sum\limits_{i=1}^{n} u_i v_i.
\]
The notation $\vecu \vec{v}$ denotes the component-wise product:
$
\vecu \vec{v} = \left[ u_i v_i \right]_{i=1}^{n},
$which will enable us to conveniently write $\left<\vecu ,\vec{v}\right>_i = \left<\vecu \vec{v}\right>_i$.
A scalar $\alpha\in\bbR$ will sometimes be identified with the vector $\vec{\alpha} = \left[\alpha\right]_{i=1}^n$.
Finally, for the norms previously defined on $U_{\#}^n(\delta\mathbb{Z})$, we will use the following notations for $\vec v\in\bbR_\per^n$:
\[
\begin{array}{c} \displaystyle
\|\vec{v}\|_{L^q(n)} = \left(\frac{1}{n}\sum\limits_{i=1}^{n} |v_i|^q\right)^{1/q}
\quad
\|\vec{v}\|_{L^{\infty}(n)} = \max\limits_{1\le i\le n} |u_i|,
\quad
|\vec{v}|_{W^{1,q}(n)} = \|D\vec{v}\|_{L^q(n)},
\\ \displaystyle
|\vec{v}|_{H^1(n)} = |\vec{v}|_{W^{1,2}(n)},
\quad
|\vec{v}|_{H^2(n)} = \|D^2 \vec{v}\|_{L^2(n)},
\quad
|\vec{v}|_{H^{-1}(n)} = \sup\limits_{\substack{\vec{w}\in \bbR_{\#}^n \\ \vec{w}\ne 0}} \frac{\left<\vec{v},\vec{w}\right>_i}{|\vec{w}|_{H^1(n)}}.
\end{array}
\]
When it will cause no confusion, we will omit the argument $n$ in the norms, thus writing only $\|\vec{v}\|_{L^2}$, $|\vec{v}|_{H^1}$, etc.

\section{Problem Formulation}\label{sec:problem_formulation}
The focus of the present study is on correct treatment of atomistic materials with spatially oscillating or inhomogeneous local properties.
For simplicity, we will first consider the 1D periodic case (the 2D case will be discussed in Section \ref{sec:2d}).

\subsection{Equations of Equilibrium}

We describe the formulation of the problem of finding an equilibrium of an atomistic material in the 1D periodic setting.
We consider the periodic boundary conditions in order to avoid difficulties arising from presence of the boundary of the atomistic material.
Otherwise, the boundary of an atomistic material, unless properly treated, would contribute an additional error to the numerical solution, studying which is not an aim of the present work.
Nevertheless, it should be noted that the numerical method and the algorithm proposed in the present work can be applied to Dirichlet, Neumann, or other boundary conditions.

Consider a material at the microscopic scale which occupies a domain $\Omega$.
We assume that the position of the atoms in reference configuration is given by $X_i=\eps i\in\eps{\mathbb Z}\cap \Omega.$
When the material experiences a deformation the atom positions become $x_i = X_i + u_i$.
We assume that the displacements $u_i$ behave periodically with a period length $N\in\mathbb{N}$, i.e.,
\begin{equation}
\label{equ:period_displ}
\quad u_{i+N} = u_i \quad -\infty<i<\infty.
\end{equation}
Setting $u_i=u(X_i)$ (see Section \ref{rk:u_identification}) we see that $u\in U_\per^N(\eps\mathbb{Z})$.
According again to Remark \ref{rk:u_identification} we will identify $u$ with $\vecu \in\bbR_\per^N$ for the discussion which follows.

We assume that the atoms $X_i,X_j$ interact through the pairwise potential $\varphi_{i,j}$, which depends on particular atoms $i$ and $j$ thus allowing for modeling heterogeneous materials.
Due to the assumption of periodic displacements we have $\varphi_{i+N,j+N}=\varphi_{i,j}.$
The energy of atomistic interaction of the system (summed for the atoms over one period) is then
\begin{align*}
\Eint(\vecu )
=~&
\epsilon \sum_{i=1}^{N} \sum_{j=i+1}^{\infty} \varphi_{i,j}\left(\frac{x_j-x_i}{\epsilon}\right)
=
\epsilon \sum_{i=1}^{N} \sum_{j=i+1}^{\infty} \varphi_{i,j}\left((j-i) + \frac{u_j-u_i}{\epsilon}\right)
\\
=~&
\epsilon \sum_{i=1}^{N} \sum_{r=1}^{\infty} \varphi_{i,i+r}\left(r + \frac{u_{i+r}-u_i}{\epsilon}\right)
=
\epsilon \sum_{i=1}^{N} \sum_{r=1}^{\infty} \varphi_{i,i+r}\left(r + r D_r u_i\right).
\end{align*}

We assume that the potential $\varphi_{i,j}(z)$ vanishes for $|z|$ large enough, so that it is sufficient to consider at most $R$ neighboring atoms in the interaction energy:
\[
\Eint(\vecu ) =
\epsilon \sum_{i=1}^{N} \sum_{r=1}^{R} \varphi_{i,i+r}\left(r + r D_r u_i \right)
= \left<\sum\limits_{r=1}^R \vec{\Phi}_{r}\left(D_r\vecu \right)\right>_i
.
\]
where $\vec{\Phi}_{r}:{\bbR_\per^N}\to \bbR_\per^N$ are introduced in the following way:
\begin{equation}
\left(\vec{\Phi}_{r}(\vec{z})\right)_i = \varphi_{i,i+r}(r + r z_i).
\label{eq:Phi}
\end{equation}

The potential energy of the external force $\vec{f}$ is
\[
\Eext(\vecu ) = -\epsilon \sum_{i=1}^{N} f_i u_i = - \left< \vec{f}, \vecu \right>_i.
\]
The forces $f_i$ on each atom are given and considered to be independent of actual atom positions $x_i$.
For the problem to be well-posed, the sum of all forces per period is assumed to be zero, i.e., $\left<\vec{f}\right>_i = 0$.

The total potential energy of the atomistic system is then
\[
\Pi(\vecu ) = \Eint(\vecu ) + \Eext(\vecu ).
\]
In these notations the problem of finding the equilibrium configuration of atoms can be written as
\begin{equation}
\frac{\partial \Pi}{\partial u_i} = 0 \quad (i=1,2,\ldots,N).
\label{eq:original_equation}
\end{equation}
For the equations \eqref{eq:original_equation} to have a unique solution, we must additionally require that the average of $\vecu $ is zero:
\begin{equation}
\left<\vecu \right>_i = 0.
\label{eq:u-averages-to-zero}
\end{equation}

The equilibrium equations \eqref{eq:original_equation} together with the additional condition \eqref{eq:u-averages-to-zero} can be written in variational form: find $\vecu \in \bbR_\per^N$ such that
\begin{subeqnarray}
	\Pi'(\vecu ;\vec{v}) = \Eint'(\vecu ;\vec{v}) + \Eext'(\vec{v}) & = & 0
	\quad \forall \vec{v}\in \bbR_\per^N
\slabel{eq:variational_equation_generic}
\\
	\left<\vecu \right>_i & = & 0,
\slabel{eq:variational_equation_generic_average0}
\label{eq:variational_problem_generic}
\end{subeqnarray}
where
\begin{eqnarray}
\notag
\Eext'(\vec{v}) &=& - \left< \vec{f}, \vec{v} \right>_i,
\\
\Eint'(\vecu ;\vec{v})
&=& \sum\limits_{r=1}^R \left<\vec{\Phi}'_{r}\left(D_{r} \vecu \right),D_{r} \vec{v}\right>_i,
\qquad \textnormal{and}
\notag
\\
\left(\vec{\Phi}'_{r}(\vec{z})\right)_i
&=& \frac{\partial}{\partial z_i}(\vec\Phi_{r}(\vec z))_i = \frac{\partial}{\partial z_i}\varphi_{i,i+r}(r + r z_i)=r\varphi_{i,i+r}'(r + r z_i)
.
\label{eq:Psi}
\end{eqnarray}
Thus, the variational form of the problem \eqref{eq:variational_problem_generic} is
\begin{subeqnarray}
	\sum_{r=1}^{R}\left<\vec{\Phi}'_{r}\left(D_{r} \vecu \right), D_{r} \vec{v}\right>_i
	& = &
	\left< \vec{f}, \vec{v} \right>_i
	\quad \forall \vec{v}\in \bbR_\per^N
\slabel{eq:variational_equation_nonlinear}
\\
	\left<\vecu \right>_i & = & 0.
\slabel{eq:variational_equation_nonlinear_average0}
\label{eq:variational_problem_nonlinear}
\end{subeqnarray}

As written in this form, the variational equation \eqref{eq:variational_equation_nonlinear} resembles the nonlinear continuum equation
\[
\left<\Phi'\left(\frac{du}{dX}\right), \frac{dv}{dX} \right> = \left< f, v\right>.
\]

The problem \eqref{eq:variational_problem_nonlinear} is often solved with the Newton's method.
It consists in choosing the initial guess $\vecu^{(0)}$ and performing iterations to find $\vecu^{(n)}$.
For that, the equations \eqref{eq:variational_problem_nonlinear} are first linearized on the solution $\vecu^{(n)}$:
\begin{subeqnarray}
\sum_{r=1}^{R}\left<\vec{\Phi}'_{r}\left(D_{r} \vecu^{(n)}\right), D_{r} \vec{v}\right>_i \hfill
\notag \\
+ \sum_{r=1}^{R}\left< \vec{\Phi}_{r}''\left(D_{r} \vecu^{(n)}\right) D_{r} \left(\vecu^{(n+1)}-\vecu^{(n)}\right), D_{r} \vec{v}\right>_i
-\left<\vec{f}, \vec{v}\right>_i
	&=& 0
	\quad \forall \vec{v}\in \bbR_\per^N\qquad
\slabel{eq:variational_equation_nonlinear_linearized}\\
	\left<\vecu^{(n+1)}\right>_i & = & 0,
\slabel{eq:variational_equation_nonlinear_linearized_average0}
\end{subeqnarray}
and then solved for the next approximation $\vecu^{(n+1)}$ until two successive iterations give close results.
Here, according to \eqref{eq:Phi} and \eqref{eq:Psi}, $\vec{\Phi}_{r}''$ is given by
\[
\left(\vec{\Phi}_{r}^{''}(\vec{z})\right)_{i} = \frac{\partial}{\partial z_i}(\vec\Phi_{r}'(\vec z))_i = r^2 \varphi_{i,i+r}'' (r + r z_i).
\]
Notice that we have identified here the $N\times N$ diagonal Jacobian matrix $\vec{\Phi}_{r}''$ with a vector in $\bbR^N$
and used the component-wise product between two vectors (see Section \ref{sub:fct_space}).

\subsection{Linearized Model and Nearest Neighbor Interaction}

	We can linearize the problem in a neighborhood of a given displacement $\bar{u}_i$:
	\[
	(\vec{\Phi}'_{r}(D_r \vecu))_i \approx r \varphi_{i,i+r}' (r + r D_r \bar{u}_i) + r^2 \varphi_{i,i+r}'' (r + r D_r \bar{u}_i) D_r (u_i-\bar{u}_i).
	\]
	Hence upon defining
	\[
	\vec{\xi}_{r} = \left[ r \varphi_{i,i+r}' (r + r D_r \bar{u}_i) - r^2 \varphi_{i,i+r}'' (r + r D_r \bar{u}_i) D_r \bar{u}_i \right]_{i=1}^{N},
	\qquad \textnormal{and}
	\]
	\begin{equation}
	\vec{\psi}_{r} = \left[ r^2 \varphi_{i,i+r}''(r + r D_r \bar{u}_i) \right]_{i=1}^{N}
	\label{eq:linearization_of_potential},
	\end{equation}
	we can use the linearized approximation
\[
\Pi'(\vecu ;\vec{v}) \approx
\sum_{r=1}^{R}\left<\vec{\xi}_{r} + \vec{\psi}_{r} D_{r} \vecu , D_{r} \vec{v}\right>_i
-\left<\vec{f}, \vec{v}\right>_i
\]
\[
=
\sum_{r=1}^{R}\left<\vec{\psi}_{r} D_{r} \vecu , D_{r} \vec{v}\right>_i
-\left<\vec{f} - \sum_{r=1}^{R} D_{r} T^{-r}\vec{\xi}_{r}, \vec{v}\right>_i.
\]
Here we used the formula of integration by parts (Lemma \ref{lem:integration-by-parts}) and again the component-wise product for $\vec\Phi'_{r}\left(D_{r} \vec u\right)$.
We see that in the case of linear interaction, the term with $\vec{\xi}_{r}$ can be absorbed into the external force $\vec{f}$, which turns the generic equilibrium equations \eqref{eq:variational_problem_generic} to
\begin{subeqnarray}
	\sum_{r=1}^{R}\left<\vec{\psi}_{r} D_{r} \vecu , D_{r} \vec{v}\right>_i
	& = & \left<\vec{f}, \vec{v}\right>_i
	\quad \forall \vec{v}\in \bbR_\per^N
\slabel{eq:variational_equation_linear}
\\
	\left<\vecu \right>_i & = & 0.
\slabel{eq:variational_equation_linear_average0}
\label{eq:variational_problem_linear}
\end{subeqnarray}
If we further assume that only nearest neighboring atoms interact (i.e., that $R=1$), then the equations \eqref{eq:variational_problem_linear} are further simplified to
\begin{subeqnarray}
	\left<\vec{\psi} D\vecu , D\vec{v}\right>_i
	& = & \left<\vec{f}, \vec{v}\right>_i
	\quad \forall \vec{v}\in \bbR_\per^N
	\slabel{eq:variational_equation_linear-nn_i}
\\
	\left<\vecu \right>_i & = & 0,
	\slabel{eq:variational_equation_linear-nn_i_average0}
\label{eq:variational_problem_linear-nn_i}
\end{subeqnarray}
where we denote $\vec{\psi} = \vec{\psi}_1$ (i.e., $\vec{\psi} = \vec{\psi}_{r}$ for $r=1$).
It will sometimes be convenient to use a ``strong form'' of \eqref{eq:variational_problem_nonlinear} or \eqref{eq:variational_problem_linear-nn}, i.e., find $\vecu \in \bbR_\per^N$ such that
\begin{subeqnarray}
-D\left(\vec{\psi} D\vecu \right)&=&T\vec{f}
	\slabel{eq:variational_equation_linear-nns}
\\
	\left<\vecu \right>_i & = & 0,
	\slabel{eq:variational_equation_linear-nn_average0s}
\label{eq:variational_problem_linear-nns}
\end{subeqnarray}
which is derived using Lemma \ref{lem:integration-by-parts}.

\section{Homogenization of Atomistic Media}\label{sec:atm-hmg}
We come now to the main subject of this paper, the treatment of materials with heterogeneous atomistic interaction as illustrated in Figures \ref{fig:springs-heterogeneous} and \ref{fig:2d-springs}.
\begin{figure}[h!t]
\begin{center}
\includegraphics{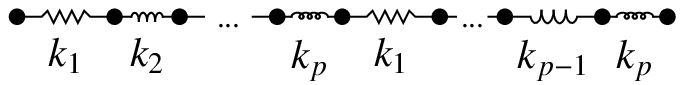}
\caption{Illustration of a 1D model problem with heterogeneous interaction.}
\label{fig:springs-heterogeneous}
\end{center}
\end{figure}

\begin{figure}
\begin{center}
\includegraphics{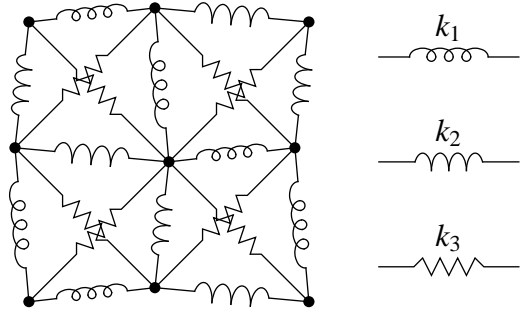}
\caption{Illustration of a 2D model problem with heterogeneous interaction.}
\label{fig:2d-springs}
\end{center}
\end{figure}

Naive coarse graining for such models (e.g., as given by the straightforward application of the quasicontinuum method)
fail to give the correct answer.
One way to treat such problems is to apply the so-called Cauchy-Born rule for complex lattices \cite{Stakgold1950, TadmorSmithBernsteinEtAl1999, SmithTadmorBernsteinEtAl2001, DobsonElliottLuskinEtAl2007}.
We present here another coarse graining strategy based on homogenization ideas.
We derive below a discrete homogenization of the atomistic material which will be the basis for formulating and analyzing a quasicontinuum method for complex lattices.
We note that our approach is different from the approach chosen in other works discussing homogenization of atomistic media
\cite{ChenFish2006, ChenFish2006a, FishChenLi2007}, which consists in treating the homogenized material at the continuum level and the heterogeneities at the atomistic level (the idea of continuous $X$ and discrete $Y$ can also be seen in the proof of the main results in \cite{EMing2007}).
In our approach, the homogenized material will retain its atomistic description.
In this section we derive the homogenized equation using asymptotic expansion.
Rigorous justification of the homogenized limit will be given in Section 4 by means of error estimates towards the full atomistic solution.

\subsection{Asymptotic expansion}\label{sec:atm-hmg:fast-and-slow}
We will assume that the local heterogeneity of the atomistic interaction is periodic with period $p\eps,~p\in\mathbb{N}$.
In order to take into account the local variation of the atomistic interaction we think of the displacement as depending on a fast and a slow scale $u(X_i)\sim u(X_i,X_i/\eps)$.
We define $X_i \in\eps\mathbb Z$, the macro (``slow'') variable, and $Y_i = X_i/\epsilon \in \mathbb Z$, the micro (``fast'') variable, and consider functions
$u^m:\eps\mathbb Z\times \mathbb Z\rightarrow \bbR$ indexed by $m=0,1,2\ldots$
As we consider periodic local interaction (with period $p\eps,~p\in\mathbb{N}$)
we assume that the functions $u^m$ are $p$-periodic in the fast variable, i.e., they satisfy
\[
u^m(X_i,Y_{j+p}) = u^m(X_i,Y_j),
\]
while the behavior w.r.t.\ $X_i$ is similar to the previously considered
\[
u^m(X_{i+N},Y_j) = u^m(X_i,Y_j).
\]
Recalling the definitions of Section \ref{sub:fct_space}, this means
$u^m\in U_\per^{N}(\eps\mathbb{Z})\otimes U_\per^{p}(\mathbb Z)$.
We then consider the asymptotic expansion
\begin{equation}
\label{equ:asympt_exp}
u= u^0(X_i, Y_j) + \epsilon u^1(X_i, Y_j) + \epsilon^2 u^2(X_i, Y_j) + \ldots
\end{equation}
In addition to the discrete derivative, translation, and integration, defined in Section \ref{sub:fct_space} we need the total derivative and the total $r$-step derivative of a function $u\in U_\per^{N}(\eps\mathbb{Z})\otimes U_\per^{p}(\mathbb Z)$:
\[
D u=
\frac{u(X_{i+1},Y_{i+1})-u(X_i,Y_i)}{\epsilon},\quad
D_r u =
\frac{u(X_{i+r},Y_{i+r})-u(X_i,Y_i)}{\epsilon r}.
\]
A simple calculation shows that the total derivative and the total $r$-step derivative can be expressed in terms of $D_X,D_Y, T_Y,D_{X,r},D_{Y,r}$, the discrete partial derivatives and translation operator defined in Section \ref{sub:fct_space}, in the following way:
\begin{eqnarray}
\label{eq:derivative-full-through-partial}
D u(X_i,Y_j)&=& D_X T_Y u(X_i,Y_j)+ \epsilon^{-1} D_Y u(X_i,Y_j),\\
\nonumber
D_{r} u(X_i,Y_j)&=& D_{X,r} T_{Y}^r u(X_i,Y_j)+ \epsilon^{-1} D_{Y,r} u(X_i,Y_j).
\end{eqnarray}

\subsection{Homogenization for Nearest Neighbor Linear Interaction}\label{sec:atm-hmg:linear-nn}

In this subsection we will perform the asymptotic analysis for the equation of equilibrium of atomistic materials.
To explain our procedure, we first treat the simplest interaction model, i.e., the case of nearest neighbor linear interaction.
Asymptotic expansion and homogenization procedure for more general cases will be given in the following subsection.
We consider the problem \eqref{eq:variational_problem_linear-nn_i}, written in functional form ($u_i=u(X_i)$):
\begin{subeqnarray}
	\left<{\psi^\eps} D u, D{v}\right>_X
	& = & \left<{f}, {v}\right>_X
	\quad \forall {v}\in U_\per^{N}(\eps\mathbb Z)
	\slabel{eq:variational_equation_linear-nn}
\\
	\left<u\right>_X & = & 0,
	\slabel{eq:variational_equation_linear-nn_average0}
\label{eq:variational_problem_linear-nn}
\end{subeqnarray}
with
$\psi^\eps$ defined as follows:
\[
\psi^\eps(X_i)= \psi(X_i, X_i/\eps)=\psi(X_i, Y_i),
\]
where the function $\psi(X_i,\bullet)\in U_\per^{p}(\mathbb Z)$, i.e., the tensor is ``$p$-periodic" in the
$Y$ variable.
We assume that the function $\psi$ is uniformly positive in the following sense:
\begin{equation}
\label{equ:coercivity_psi}
\psi(X_i,Y_j)\geq c_\psi >0
\quad\forall (X_i,Y_j)\in\eps\mathbb{Z}\times\mathbb{Z}.
\end{equation}

We also assume that the external force $f$ it does not depend on $Y$, i.e.,
$f=f(X_i)$.
We emphasize that oscillatory external forces could also be considered.
The homogenized equation would then depend on a proper average of the external forces.
For simplicity we will not consider this case.
We now proceed as in the ``classical homogenization" \cite{Bakhvalov1974, BensoussanLionsPapanicolaou1978, S'anchez-Palencia1980}
and plug the ansatz \eqref{equ:asympt_exp} in \eqref{eq:variational_equation_linear-nns} (we will go back and forth from the variational formulation \eqref{eq:variational_problem_linear-nn} to the strong formulation \eqref{eq:variational_problem_linear-nns}).
This gives
\begin{eqnarray}
\nonumber
-(T_X^{-1} D_X+\eps^{-1}D_Y)&\Big(&\psi D_X T_Yu^0+\eps^{-1}\psi D_Yu^0+\eps\psi D_X T_Yu^1+\psi D_Yu^1\\
\label{equ:plug_ansatz}
&+&\eps^2\psi D_X T_Yu^2+\eps D_Y u^2+\ldots\Big)=f.
\end{eqnarray}
Here we used the identity \eqref{eq:derivative-full-through-partial}, and Lemma \ref{lem:integration-by-parts} to compute the adjoint $(D_X T_Y + \epsilon^{-1} D_Y)^* = (T_X^{-1} D_X T_Y^{-1} + \epsilon^{-1} T_Y^{-1} D_Y)$.
We thus obtain a cascade of equations and collect powers of $\eps$.

Collect the $O(\epsilon^{-2})$ terms in \eqref{equ:plug_ansatz}:
\begin{eqnarray*}
-D_Y\left(\psi D_Y u^0\right)= 0\\
u^0\hbox{ is $p$-periodic in Y.}
\end{eqnarray*}
Thanks to \eqref{equ:coercivity_psi} we have $D_Y u^0=0$.
This implies that $u^0$ is independent of $Y$ and only a function of $X$, i.e.,
\[
u^0(X_i, Y_j) = u^0(X_i).
\]

We next collect the $O(\epsilon^{-1})$ terms in \eqref{equ:plug_ansatz}:
\begin{subeqnarray}
\label{equ:u1}
\slabel{equ:u1:a}
-D_Y\left(\psi D_Y u^1\right)=D_Y(\psi D_X u^0)
\\
\slabel{equ:u1:b}
u^1\hbox{ is $p$-periodic in Y,}
\end{subeqnarray}
where we have used the fact that $u^0$ does not depend on $Y$, which implies $D_X T_Y(\psi D_Yu^0)=0$ and $T_Y u^0=u^0$.
As usual in homogenization we take advantage of the separation of variables of the right hand side of \eqref{equ:u1:a} and we let $\chi=\chi(X_i; Y_j)$
be the solution of
\begin{subeqnarray}
\label{equ:chi}
\slabel{equ:chi:a}
-D_Y\left(\psi D_Y \chi\right)=D_Y\psi\\
\chi^1\hbox{ is $p$-periodic in Y.}
\end{subeqnarray}
In view of \eqref{equ:coercivity_psi}, this problem has a unique solution (up to an additive constant) if and only if $\left<D_Y\psi\right>=0$ (solvability condition) which indeed holds due to the periodicity assumption on $\psi$.
Existence and uniqueness follows from the Lax-Milgram theorem for the following variational problem (see Lemma \ref{lem:cell-problem}): find $\chi(X_i,\bullet)\in U_{\#}^{p}(\mathbb Z)$ such that
\begin{equation}
\left<\psi D_Y \chi, D_Y s\right>_Y = -\left<\psi, D_Y s\right>_Y
\quad \forall s=s(Y_j)\,\in U_{\#}^{p}(\mathbb Z).
\label{eq:linear-nn_chi}
\end{equation}
It is then readily seen that $u^1(X_i, Y_j) = \chi(X_i; Y_j) D_X u^0(X_i)$ solves \eqref{equ:u1}.
The general solution of this latter equation involves a constant depending on $X_i$
determined by the condition $\left<\chi(X_i; \bullet)\right>_Y = 0$ (recall that functions in the space $U_{\#}^{p}(\mathbb Z)$
have zero average), i.e.,
\[
u^1(X_i, Y_j) =\chi(X_i; Y_j) D_X u^0(X_i) + \baru^1(X_i).
\]

Finally, collecting the $O(\epsilon^{0})$ terms in \eqref{equ:plug_ansatz}
gives
\begin{eqnarray*}
&& -D_Y\left(\psi D_Y u^2\right)
	=
	D_Y\left(\psi D_X T_Y u^1\right)
	+ T_X^{-1} D_X\left(\psi(1+D_Y\chi)D_X u^0\right) + f\\
& & u^2\hbox{ is $p$-periodic in Y.}
\end{eqnarray*}
The solvability condition for the existence of a solution $u^2$ reads
\[
\left<D_Y\left(\psi D_X T_Y u^1\right) + T_X^{-1}D_X\left(\psi(1+D_Y\chi)D_X u^0\right)+f\right>_Y=0,
\]
leading to the homogenized equation
\begin{subeqnarray}
-D_X\left(\psi^0D_X u^0\right)&=&T_X f
	\slabel{eq:linear-homogenized-nn}
\\
	\left<{u^0}\right>_X & = & 0,
	\slabel{eq:linear-homogenized-nn_average0}
\label{eq:linear-nn_homogenized}
\end{subeqnarray}
where we choose, as for the original problem \eqref{eq:variational_problem_linear-nn}, the periodic boundary conditions and where
\begin{equation}
\psi^0 = \left<\psi \left(1 + D_Y \chi\right)\right>_Y.
\label{eq:linear-nn_homogenized-tensor}
\end{equation}

Thus, we obtained the equation for the homogenized displacement $u^0$ with the homogenized discrete tensor $\psi^0$.
The homogenized tensor $\psi^0$ no longer depends on the fast variable $Y$ and therefore we can apply the standard QC method to the homogenized equation \eqref{eq:linear-nn_homogenized}.
The equation \eqref{eq:linear-nn_homogenized} has to be supplemented with boundary conditions.
Our choice of periodic boundary conditions for the displacement (see \eqref{equ:period_displ})
leads to searching for $u^0\in U_\per^{N}(\eps\mathbb{Z})$.

In the simple case of nearest neighbor linear interaction, the homogenized discrete tensor $\psi^0 = \left<\psi \left(1 + D_Y \chi\right)\right>_Y$ can be found analytically.
Indeed, from \eqref{equ:chi:a} we see that
$\psi (1 + D_Y \chi)$ does not depend on $Y$:
\begin{equation}
\psi(Y_j) (1 + D_Y \chi(X_i; Y_j)) = C(X_i),
\label{eq:linear-nn_exact_first_integral}
\end{equation}
from where we find
\begin{equation}
D_Y \chi = \frac{C}{\psi} - 1.
\label{eq:linear-nn_exact_DYchi}
\end{equation}
The constant of integration $C = C(X_i)$ can be found by averaging \eqref{eq:linear-nn_exact_DYchi} over $Y$:
\[
0 = \left<C/\psi - 1\right>_Y = C \left<1/\psi\right>_Y - 1,
\]
from where we find
\begin{equation}
C = \left<1/\psi\right>_Y^{-1},
\label{eq:linear-nn_exact:constC}
\end{equation}
and the homogenized tensor is thus
\begin{equation}
\psi^0 = \left<\psi \left(1 + D_Y \chi\right)\right>_Y
=
\left<\psi \left(1 + \frac{C}{\psi} - 1\right)\right>_Y = \left<C\right>_Y = C = \left<1/\psi\right>_Y^{-1}.
\label{eq:linear-nn_exact:homogenized-tensor}
\end{equation}
Thus the homogenized equations \eqref{eq:linear-nn_homogenized} are written as
\[
D_X\left(\left<1/\psi\right>_Y^{-1}\right) D_X u^0=f.
\]
We emphasize that this procedure and the obtained results are well-known for PDEs \cite[Chap.\ 1]{BensoussanLionsPapanicolaou1978}.

\subsection{Generalizations}\label{sec:atm-hmg:generalizations}

Below we generalize the results of the previous subsection to the cases of finite range (i.e., $R>1$) linear (Section \ref{sec:atm-hmg:generalizations:general-linear}) and nonlinear (Section \ref{sec:atm-hmg:general-nonlinear}) interaction, omitting this details of technical nature.

\subsubsection{Finite Range Linear Interaction}\label{sec:atm-hmg:generalizations:general-linear}

One technical difficulty in this case is that there are $R$ different differentiation operators $D_{r}$.
Then a straightforward generalization of the results of the previous subsection would yield $u^1(X_i, Y_j)$ depending on $R$ discrete ``macroscale'' derivatives $D_{X,r} u^0$ ($1\le r\le R$).
This approach would also essentially differ from the results in continuum homogenization.
Therefore, instead of the identity
\begin{equation}
D_{r} = D_{X,r} + \epsilon^{-1} D_{Y,r}
\label{eq:linear_Di}
\end{equation}
the following approximate (accurate to $O(\epsilon)$) identity should be used:
\begin{equation}
D_{r} \simeq T_{Y}^r D_X + \epsilon^{-1} D_{Y,r}.
\label{eq:linear_Di_approximate}
\end{equation}
The accuracy of $O(\epsilon)$ is enough if one seeks to obtain only the homogenized solution $u^0(X_i)$ and the leading term of the correction $\epsilon \chi(X_i; Y_j) D_X u^0(X_i)$.

The procedure can now be performed similarly to the nearest neighbor interaction case: we plug the ansatz
\eqref{equ:asympt_exp} in \eqref{eq:variational_equation_linear}.
As previously, we obtain that $u^0(X_i,Y_j) = u^0(X_i)$ and the $O(\epsilon^{-1})$ terms yields
\[
u^1(X_i,Y_j) = \chi(X_i; Y_j) D_X u^0(X_i) + \baru^1(X_i),
\]
where $\chi$ is a solution of
\begin{subeqnarray}
\label{equ:chir}
-\sum_{r=1}^{R} D_{Y,r}\left(\psi_r D_{Y,r} \chi\right) = \sum_{r=1}^{R} D_{Y,r}\psi_r\\
\chi\hbox{ is $p$-periodic in Y.}
\end{subeqnarray}
Appropriate conditions on $\psi_r$ are required to ensure that \eqref{equ:chir} has a unique solution.
Collecting the $O(\epsilon^0)$ terms using the solvability conditions for $u^2$ (similarly to the nearest neighbor case)
yields the homogenized equation
\[
-D_X\left(\psi^0D_X u^0\right)=T_X f,
\]
where the homogenized tensor $\psi^0$ is defined as
\[
\psi^0 = \sum_{r=1}^{R} \left<\psi \left(1 + D_{Y,r} \chi\right)\right>_Y.
\]

A remarkable feature of the result of homogenization of material with finite range interaction is that the material thus homogenized contains only the nearest neighbor interaction despite the original model having longer interactions.
This is a consequence of our choice of $D$ in the form \eqref{eq:linear_Di_approximate}.
If we would have chosen the exact relation \eqref{eq:linear_Di}, then the homogenized material would contain the same number of interacting atoms.

\subsubsection{Finite Range Nonlinear Interaction}\label{sec:atm-hmg:general-nonlinear}

In this subsection we further generalize the results to the case of a general nonlinear material \eqref{eq:variational_problem_nonlinear}
\begin{eqnarray*}
-\sum_{r=1}^{R}D_{-r}\left[(\Phi^{\epsilon}_r)'\left(D_{r} \vecu \right)\right]&=&f\\
	\left<\vecu \right>_i & = & 0.
\end{eqnarray*}
We assume that the nonlinear tensor $(\Phi^{\epsilon}_r)'(z)_i$ has the form $\Phi_r'(z; X_i,X_i/\eps)$ and set
$Y_i=X_i/\eps$ as previously.
In accordance with the definition \eqref{eq:Psi} it means that the interacting potential $\varphi_{i,j}$ depends on $X_i$ and $Y_j$
and
\[
\Phi'_{r}(z;X_i,Y_j) = r \frac{\partial \varphi_{r}}{\partial z}(r+r z; X_i,Y_j).
\]
We again proceed with the asymptotic expansion.
We use the ansatz \eqref{equ:asympt_exp} (we directly assume that $u^0=u^0(X_i)$ in order to simplify the argument), the approximation \eqref{eq:linear_Di_approximate}
in the above nonlinear equation and identify the power of $\eps$.
This yields the homogenized equation
\[
-D_X\left[(\Phi^0)'\left(D_X u^0\right)\right]=T_X f,
\]
where
\[
(\Phi^0)'(z) = \sum_{r=1}^{R} \left<\Phi'_{r}\left(z+ D_{Y,r} \chi(z)\right)\right>_Y.
\]
The function $\chi(z)=\chi(z;X_i,Y_j)$ solves the parametric problem
\begin{subeqnarray}
\label{equ:chir_nonlin}
-\sum_{r=1}^{R} D_{Y,r}\left[(\Phi^\epsilon_r)'(z+D_{Y,r} \chi(z))\right]=0\\
\chi\hbox{ is $p$-periodic in Y.}
\end{subeqnarray}
Of course, structure assumptions on $(\Phi^\epsilon_r)'$ are needed to ensure a unique solution of \eqref{equ:chir_nonlin}.
We will not go into details here as our analysis in Section 4 and 5 will be limited to the linear case.

\begin{remark}
It is useful to highlight one more feature of the homogenized equations, before we proceed with describing the numerical algorithm.
In terms of atomistic interaction potential $\varphi^\epsilon_{r}(z) = \varphi_{r}(z;X_i, Y_j)$ the homogenized tensor $(\Phi^0)'$ is
\[
(\Phi^0)'(D_X u^0) = \sum_{r=1}^{R} \left<r \frac{\partial \varphi_{r}}{\partial z}\left(r + r (D_X u^0 + D_{Y,r} \chi)\right)\right>_Y.
\]
This representation can be interpreted as the averaging of the functional derivative of the original energy
\[
\sum\limits_{r=1}^{R} \varphi_{r}\left(r+r (D_X u^0 + D_{Y,r} \chi)\right)
\]
at the corrected solution $D_X u^0 + D_{Y,r} \chi$.
This fact is important in showing the equivalence between QC applied to homogenized material and QC for complex lattices \cite{TadmorSmithBernsteinEtAl1999}
and will be proved in the companion paper \cite{AbdulleLinShapeevII}.
\end{remark}

\section{Analysis of Equations}\label{sec:estimates}

In this section we show that the original and the homogenized problems of equilibrium of materials with spatially oscillating properties are well-posed and that the difference between their solutions is $O(\epsilon)$ in the appropriate norms.
We limit our analysis to the case of nearest neighbor linear interaction in the 1D periodic setting, but allow the material properties to vary.
Such interaction corresponds to the nonlinear interaction linearized on a given non-uniform deformation.

After defining the appropriate norms for measuring the error (Section \ref{sec:estimates:preliminaries}), we state the main theorems (Section \ref{sec:estimates:main_results}) followed by proof of technical lemmas (Section \ref{sec:estimates:technical_results}).

In this section, by $C_0$, $C_1$, $C_2$, $C_3$ we denote generic constants which may depend on $c_{\psi}$, $C_{\psi}$, $C_{\psi}'$, and $p$, but are independent of $\epsilon$.

\subsection{Preliminaries}\label{sec:estimates:preliminaries}

Let $u\in U_{\#}^{N}(\eps\mathbb{Z})$ be the solution of \eqref{eq:variational_problem_linear-nn}.
We assume as in the previous section that the tensor $\psi^\eps$ can be written as
\begin{equation}
\psi^\eps(X_i)= \psi(X_i, X_i/\eps)=\psi(X_i, Y_i),
\label{eq:psi-multiscale-assumption}
\end{equation}
where the function $\psi(X_i,\bullet)\in U_\per^{p}(\mathbb Z)$ (i.e., is ``$p$-periodic" in the
$Y$ variable).
This holds, for instance, if we linearize the interaction the original nonlinear model on the displacement $u^0$ (cf.\ \eqref{eq:linearization_of_potential}) that can be expressed as $u^0 = u^0(X_i, X_i/\eps)$.
We also assume that $\psi^\eps$ satisfies
\begin{eqnarray}
\label{equ:coerc_bound_psi}
& & 0<c_{\psi}\le \psi(X_i, Y_j)\le C_{\psi}
\quad \forall X_i\in\eps\mathbb Z, Y_j\in\mathbb Z,\\
\label{equ:bound_dpsi}
& & \|D_X \psi\|_{L^{\infty}(N,p)}\le C'_{\psi}.
\end{eqnarray}

Let $u^0\in U_{\#}^{N}(\eps\mathbb{Z})$ be the solution of \eqref{eq:linear-nn_homogenized}
where the homogenized tensor $\psi^0$ is given by
\[
\psi^0(X_i) = \left<1/\psi(X_i,\bullet)\right>_Y^{-1},
\]
and $\chi(X_i,\bullet)\in U_\per^{p}(\mathbb Z)$ is a solution of \eqref{equ:chi}.
It clearly follows from \eqref{equ:coerc_bound_psi}
that $\psi^0$ is also coercive and bounded, i.e.,
\begin{eqnarray}
\label{equ:coerc_bound_psi0}
& & 0<c_{\psi}\le \psi^0(X_i)\le C_{\psi}
\quad \forall X_i\in\eps\mathbb Z,
\end{eqnarray}

The existence and uniqueness of a solution of \eqref{eq:variational_problem_linear-nn}
\eqref{eq:linear-nn_homogenized-tensor}, and \eqref{equ:chi} follow from standard arguments.
For the sake of completeness we briefly sketch the proof.
\begin{proposition}
\label{prop:original-problem}
Let $\psi$ satisfy \eqref{equ:coerc_bound_psi} and assume $\left<{f}\right>_X=0$.
Then the problems \eqref{eq:variational_problem_linear-nn}
and \eqref{eq:linear-nn_homogenized-tensor} have unique solutions $u,u^0\in U_{\#}^{N}(\eps\mathbb{Z})$
respectively, and the following estimates hold
\begin{eqnarray}
|u|_{H^1} \le c_{\psi}^{-1}|{f}|_{H^{-1}}
,
\label{eq:original-problem:estimate}\\
|{u^0}|_{H^1} \le c_{\psi}^{-1}|{f}|_{H^{-1}}
.
\label{eq:homogenized-problem:estimate}
\end{eqnarray}
\end{proposition}
\begin{proof}
We first notice that, thanks to the condition $\left<{f}\right>_X=0$, $\left<{f},\bullet\right>_X$
is a linear form on $U_{\#}^{N}(\eps\mathbb{Z})$.
Problem \eqref{eq:variational_problem_linear-nn}
can then be rewritten as follows: find $u\in U_{\#}^{N}(\eps\mathbb{Z})$ such that
\[
	\left<{\psi^\eps} D u, D{v}\right>_X
	= \left<{f}, {v}\right>_X
	\quad \forall {v}\in U_{\#}^{N}(\eps\mathbb Z).
\]
Using \eqref{equ:coerc_bound_psi} we have $\left<{\psi} Du, Du\right>_X\ge c_{\psi} |u|_{H^1}^2$
and the Lax-Milgram theorem concludes the proof.
The proof of \eqref{eq:homogenized-problem:estimate} follow the lines of the above proof using \eqref{equ:coerc_bound_psi0}.
\end{proof}
\begin{proposition}
\label{lem:chi_exist}
Let $\psi$ satisfy \eqref{equ:coerc_bound_psi}.
Then \eqref{equ:chi}
has a unique solution $\chi(X_i,\bullet)\in U_{\#}^{p}(\mathbb Z)$.
Moreover, $\chi\in U_\per^{N}(\eps\mathbb{Z})\otimes U_{\#}^{p}(\mathbb{Z})$.
\end{proposition}
\begin{proof}
The problem \eqref{equ:chi} can be written as follows: find $\chi(X_i,\bullet)\in U_{\#}^{p}(\mathbb Z)$
such that
\begin{equation}
\label{equ:chi_variat}
	\left<{\psi(X_i,\bullet)} D_Y\chi, D_Y s\right>_Y
	= -\left<D_Y{\psi(X_i,\bullet)},s\right>_Y
	\quad \forall {s}\in U_{\#}^{p}(\mathbb Z).
\end{equation}
As $\psi$ is $p$-periodic in the $Y$ variable, we have $\left<D_Y{\psi(X_i,\bullet)}\right>_Y=0$
and the existence and uniqueness of a solution (depending on $X_i$) can be established as in Proposition \ref{prop:original-problem}.
Notice that $\chi$ depends on $X_i$.
As the equation
\eqref{equ:chi_variat} remains unchanged when $\psi(X_i,\bullet)$ is changed to $\psi(X_{i+N},\bullet)$, we also have $\chi(\bullet,Y_i)\in U_\per^{N}(\eps\mathbb{Z})$.
\end{proof}

Define now the corrector
\begin{equation}
\label{equ:corrector}
u^{\rm c}(X_i)= u^0(X_i) + \epsilon\chi(X_i,X_i/\eps) D_X u^0(X_i).
\end{equation}
In the following subsection we show that $\left|u^{\rm c}-u\right|_{H^1}
\le C_1 \epsilon\|{f}\|_{L^2}$ (Theorem \ref{thm:apriory-H1})
and $\left\|u^0-u\right\|_{L^2} \le C_3 \epsilon \left\|{f}\right\|_{L^2}$ (Theorem \ref{thm:apriory-L2}).

\subsection{Main results}\label{sec:estimates:main_results}
We start with formulating the following two technical lemmas that will be proved in Section \ref{sec:estimates:technical_results}.

\begin{lemma}\label{lem:cell-problem}
Let $\chi$ be the solution of \eqref{equ:chi}.
\begin{itemize}
\item[(a)]
If \eqref{equ:coerc_bound_psi} holds then
\begin{equation}
\|\chi\|_{L^{\infty}(N,p)}\le \frac{p}{2} \frac{C_{\psi}}{c_{\psi}}.
\label{eq:Xchi-estimate}
\end{equation}
\item[(b)]
If both \eqref{equ:coerc_bound_psi} and \eqref{equ:bound_dpsi} hold then
\begin{equation}
|\psi^0|_{W^{1,\infty}(N)}\le \frac{C_{\psi}}{c_{\psi}} C'_{\psi}
,
\qquad \textnormal{and}
\label{eq:psi0_W1infty}
\end{equation}
\begin{equation}
\|D_X \chi\|_{L^{\infty}(N,p)}\le p \frac{C'_{\psi}}{c_{\psi}}.
\label{eq:DXchi-estimate}
\end{equation}
\end{itemize}
\end{lemma}

In what follows, a function of two variables (e.g., $\chi=\chi(X_i, Y_j)$) may be identified with a corresponding function of one variable ($\chi=\chi(X_i, X_i/\epsilon)$).
	Whenever it may cause confusion we will explicitly specify the function space or the norm (i.e., $\left\|\chi\right\|_{L^2(N,p)}$ or $\left\|\chi\right\|_{L^2(N)}$).
\begin{lemma}\label{lem:aposteriory}
Let $u^0\in U_{\#}^{N}$ and $\chi\in U_\per^{N}(\eps\mathbb{Z})\otimes U_{\#}^{p}(\mathbb{Z})$ be the solutions of \eqref{eq:linear-nn_homogenized}
and \eqref{equ:chi}, respectively.
Assume that \eqref{equ:coerc_bound_psi} holds and that $N/p\in\mathbb N$.
Then the corrector
$u^{\rm c}$ defined in \eqref{equ:corrector} belongs to $U_\per^{N}(\eps\mathbb Z)$
and its average is estimated as
\begin{equation}
\left|\left<u^{\rm c}\right>_X\right| \le \epsilon^2 \frac{p}{2} \left\|D_X(\chi D_X u^0)\right\|_{L^1(N,p)}.
\label{eq:aposteriory:uc-average}
\end{equation}
Furthermore, the following estimate holds:
\begin{equation}
\left|u^{\rm c}-u\right|_{H^1}
\le
\epsilon \frac{C_{\psi}}{c_{\psi}} \left\|D_X (T_Y\chi D_X u^0)\right\|_{L^2(N)}.
\label{eq:aposteriory:estimateH1}
\end{equation}
\end{lemma}

\begin{theorem}
\label{thm:apriory-H1}
Assume that $\left<f\right>_X=0$, $N/p\in\mathbb N$, and that \eqref{equ:coerc_bound_psi} and \eqref{equ:bound_dpsi} hold.
Then there exist constants $C_1$, $C_2$ such that
\begin{eqnarray}
\left|u^{\rm c}-u\right|_{H^1}
& \le & \epsilon C_1 \|{f}\|_{L^2}
,
\qquad\textnormal{and}
\label{eq:apriory-H1:norm-estimate}
\\
\left|\left<u^{\rm c}\right>_X\right|
& \le &
\epsilon^2 C_2 \|{f}\|_{L^2},
\label{eq:apriory-H1:average-estimate}
\end{eqnarray}
where $u^{\rm c}$ is the corrector defined in \eqref{equ:corrector}.
\end{theorem}
\begin{proof}
To show \eqref{eq:apriory-H1:norm-estimate} we need to estimate the right-hand side of \eqref{eq:aposteriory:estimateH1}:
\begin{align*}
\epsilon^{-1} \left|u^{\rm c}-u\right|_{H^1}
\le~&
\frac{C_{\psi}}{c_{\psi}} \left\|D_X \left(\chi(X_i; Y_{i+1}) D_X u^0(X_i)\right)\right\|_{L^2(N)}
\\ \le~&
\frac{C_{\psi}}{c_{\psi}} \left\|\left(D_X \chi(X_i; Y_{i+1})\right)D_X u^0(X_i)\right\|_{L^2(N)}
\\ ~& +
\frac{C_{\psi}}{c_{\psi}} \left\|\chi(X_{i+1}; Y_{i+1}) D_X^2 u^0(X_i)\right\|_{L^2(N)}
\\ \le~&
\frac{C_{\psi}}{c_{\psi}}\,p \frac{C'_{\psi}}{c_{\psi}}\,c_{\psi}^{-1}|{f}|_{H^{-1}}
+
\frac{C_{\psi}}{c_{\psi}}\,\frac{p}{2} \frac{C_{\psi}}{c_{\psi}} C_0 \|{f}\|_{L^2}
\\ \le~&
\frac{C_{\psi}}{c_{\psi}}\,p \frac{C'_{\psi}}{c_{\psi}} c_{\psi}^{-1} \frac{1}{2\sqrt{3}} \|{f}\|_{L^2}
+
\frac{C_{\psi}}{c_{\psi}}\,\frac{p}{2} \frac{C_{\psi}}{c_{\psi}} C_0 \|{f}\|_{L^2},
\end{align*}
where we used \eqref{eq:DXchi-estimate}, \eqref{eq:homogenized-problem:estimate}, and \eqref{eq:DXchi-estimate}
to estimate $D_X \chi,D_X u^0$, and $\chi$, respectively, and also \eqref{eq:Hminus1-through-L2} (Lemma \ref{lem:Hminus1-through-L2}) to estimate $|{f}|_{H^{-1}}$ through $\|{f}\|_{L^2}$.
Notice that we used the estimate
\[
\left\|D_X^2 u^0(X_i)\right\|_{L^2(N)} \le C_0 \|{f}\|_{L^2},
\]
which can be obtained with the help of \eqref{equ:coerc_bound_psi0}:
\[
\left\|D^2 u^0\right\|_{L^2}
\le c_{\psi}^{-1} \left\|{\psi}^0 D^2 u^0\right\|_{L^2}
=c_{\psi}^{-1} \left\|D\left({\psi}^0 D u^0\right)-(D{\psi}^0) (D u^0)\right\|_{L^2},
\]
by estimating the terms $\left\|D\left({\psi}^0 D u^0\right)\right\|_{L^2}$, $\|D{\psi}^0\|_{L^\infty}$, and $\left\|D u^0\right\|_{L^2}$ using
\eqref{eq:linear-homogenized-nn}, \eqref{eq:psi0_W1infty}, and \eqref{eq:original-problem:estimate}, respectively.

To show \eqref{eq:apriory-H1:average-estimate} we need to estimate the right-hand side of \eqref{eq:aposteriory:uc-average}:
\begin{align*}
\epsilon^{-2} \left|\left<u^{\rm c}\right>_i\right|
\le~&
\frac{p}{2} \left\|D_X(\chi D_X u^0)\right\|_{L^1(N,p)}
\le
\frac{p}{2} \left\|D_X(\chi D_X u^0)\right\|_{L^2(N,p)}
\\ \le~&
\frac{p}{2} \left\|(D_X \chi) D_X u^0\right\|_{L^2(N,p)}
+
\frac{p}{2} \left\|(T_X \chi) D^2_X u^0\right\|_{L^2(N,p)}
\\ \le~&
\frac{p}{2}\,p \frac{C'_{\psi}}{c_{\psi}}\,c_{\psi}^{-1} \left|{f}\right|_{H^{-1}}
+
\frac{p}{2}\,\frac{p}{2}\frac{C_{\psi}}{c_{\psi}} C_0 \left\|{f}\right\|_{L^2}
\\ \le~&
\frac{p}{2}\, p \frac{C'_{\psi}}{c_{\psi}}\,c_{\psi}^{-1} \frac{1}{2\sqrt{3}} \left\|{f}\right\|_{L^2}
+
\frac{p}{2}\,\frac{p}{2}\frac{C_{\psi}}{c_{\psi}} C_0 \left\|{f}\right\|_{L^2}.
\end{align*}
\end{proof}

\begin{theorem}
\label{thm:apriory-L2}
Assuming the hypotheses of Theorem \ref{thm:apriory-H1}, there exists a constant $C_3$ such that
\[
\left\|u^0-u\right\|_{L^2} \le C_3 \epsilon \left\|{f}\right\|_{L^2},
\]
\end{theorem}
\begin{proof}
Using Theorem \ref{thm:apriory-H1} yields:
\begin{align*}
\left\|u^0-u\right\|_{L^2}
\le~&
\left\|u^0-u^{\rm c}\right\|_{L^2}
+
\left\|u^{\rm c}-u\right\|_{L^2}
\le
\left\|\epsilon{\chi} D_X u^0\right\|_{L^2}
+
\frac{1}{2\sqrt{3}}\left|u^{\rm c}-u\right|_{H^1}
\\ \le~&
\epsilon\,\frac{p}{2} \frac{C_{\psi}}{c_{\psi}} c_{\psi}^{-1}\|{f}\|_{H^{-1}}
+
\epsilon \frac{C_1}{2\sqrt{3}} \left\|{f}\right\|_{L^2}
\\ \le~&
\epsilon\,\frac{p}{2} \frac{C_{\psi}}{c_{\psi}} c_{\psi}^{-1} \frac{1}{2\sqrt{3}} \left\|{f}\right\|_{L^2}
+
\epsilon \frac{C_1}{2\sqrt{3}} \left\|{f}\right\|_{L^2}.
\end{align*}
\end{proof}

\subsection{Proof of Technical Lemmas}\label{sec:estimates:technical_results}

\begin{proof}[Proof of lemma \ref{lem:cell-problem}.]
In a straightforward, but very tedious calculation, one can derive, using \eqref{eq:linear-nn_exact_DYchi}, \eqref{eq:linear-nn_exact:constC}, and \eqref{eq:linear-nn_exact:homogenized-tensor}, the exact representation
\[
\chi(X_i, Y_j) = \left<1/\psi\right>_Y^{-1} \sum_{\beta=j-p}^{j-1} \frac{p+1 - 2 (j-\beta)}{2 p} \frac{1}{\psi(X_i, Y_{\beta})}
=
\psi^0(X_i) \left<\frac{g(Y_j-\bullet)}{\psi(X_i, \bullet)} \right>_Y
,
\]
where $g\in U_{\#}^p$ is defined as $g(Y_j) = \frac{p+1}{2} - j$ for $1\le j\le p$.
Hence \eqref{eq:Xchi-estimate} holds:
\[
\left|\chi(X_i, Y_j)\right|
\le
\psi^0(X_i)\, \|g(X_i, \bullet)\|_{L^{\infty}(p)}\, \left<\frac{1}{\psi(X_i, \bullet)} \right>_Y
\le
C_{\psi} \, \frac{p}{2} \, \frac{1}{c_{\psi}}
.
\]

To show \eqref{eq:psi0_W1infty} notice that
\begin{align*}
D_X \psi^0(X_i)
=~&
D_X \left<\frac{1}{\psi(X_i, \bullet)}\right>_Y^{-1}
\\ =~&
	- \left<\frac{1}{\psi(X_i, \bullet)}\right>_Y^{-1}
	\left<\frac{1}{\psi(X_{i+1}, \bullet)}\right>_Y^{-1}
	D_X \left<\frac{1}{\psi(X_i, \bullet)}\right>_Y
\\
=~&
	\psi^0(X_i) \psi^0(X_{i+1})
	\left<\frac{D_X \psi(X_i, \bullet)}{\psi(X_i, \bullet) \psi(X_{i+1}, \bullet)}\right>_Y
,
\end{align*}
and hence if we additionally assume \eqref{equ:bound_dpsi} then
\begin{align*}
\left|D_X \psi^0(X_i)\right|
\le ~&
	\psi^0(X_i) \psi^0(X_{i+1})
\left|
	\left<\frac{1}{\psi(X_{i+1}, \bullet)}\right>_Y
\right| \frac{C'_{\psi}}{c_{\psi}}
\\ =~&
	\frac{\psi^0(X_i) \psi^0(X_{i+1})}{\psi^0(X_{i+1})}
	\frac{C'_{\psi}}{c_{\psi}}
\le
	\frac{C_{\psi}}{c_{\psi}} C'_{\psi}.
\end{align*}

To show \eqref{eq:DXchi-estimate} notice that
\begin{align*}
D_X \chi(X_i, Y_j)
=~&
	D_X \psi^0(X_i)
	\left<g(Y_j-\bullet) \frac{1}{\psi(X_i, \bullet)}\right>_Y
\\ ~&
	+
	\psi^0(X_{i+1})
	\left<g(Y_j-\bullet) D_X \frac{1}{\psi(X_i, \bullet)}\right>_Y
\\ = ~&
	\psi^0(X_i) \psi^0(X_{i+1}) \left<\frac{D_X \psi(X_i,\bullet)}{\psi(X_i, \bullet) \psi(X_{i+1}, \bullet)}\right>_Y
	\left<g(Y_j-\bullet) \frac{1}{\psi(X_i, \bullet)}\right>_Y
\\ ~&
	+
	\psi^0(X_{i+1})
	\left<g(Y_j-\bullet) \frac{D_X \psi(X_i,\bullet)}{\psi(X_i, \bullet) \psi(X_{i+1}, \bullet)}\right>_Y
,
\end{align*}
and use \eqref{equ:coerc_bound_psi}, \eqref{equ:bound_dpsi}, and $|g|\le \frac{p}{2}$ to estimate
\begin{align*}
\left|D_X \chi(X_i, Y_j)\right|
\le~&
	\psi^0(X_i) \psi^0(X_{i+1}) \frac{C'_{\psi}}{c_{\psi}} \left<\frac{1}{\psi(X_{i+1}, \bullet)}\right>_Y
	\frac{p}{2} \left<\frac{1}{\psi(X_i, \bullet)}\right>_Y
\\ ~& +
	\psi^0(X_{i+1})
	\frac{p}{2} \frac{C'_{\psi}}{c_{\psi}} \left<\frac{1}{\psi(X_{i+1}, \bullet)}\right>_Y
\\ = ~&
	\frac{C'_{\psi}}{c_{\psi}} \frac{p}{2}
+
	\frac{p}{2} \frac{C'_{\psi}}{c_{\psi}}
=	p \frac{C'_{\psi}}{c_{\psi}}.
\end{align*}
\end{proof}

\begin{proof}[Proof of lemma \ref{lem:aposteriory}.]
Under the condition $N/p\in\mathbb N$ it immediately follows from \eqref{equ:corrector} that $u^{\rm c}(X_{i+N}) = u^{\rm c}(X_i)$, hence $u^{\rm c}\in U_\per^{N}(\eps\mathbb Z)$.

Denote $v = \chi(X_i,Y_j) D_X u^0(X_i)$, so that $u^{\rm c} = u^0(X_i) + \epsilon v(X_i, X_i/\epsilon)$.
Since $\left<\chi\right>_Y=0$,
\[
0 = \left<v\right>_{XY}
=
\frac{\epsilon}{p}\sum\limits_{i=1}^N \sum\limits_{j=-\lceil p/2\rceil+1}^{\lfloor p/2\rfloor}
	v(X_i, X_i/\epsilon-j)
=
\frac{\epsilon}{p}\sum\limits_{i=1}^N \sum\limits_{j=-\lceil p/2\rceil+1}^{\lfloor p/2\rfloor}
	v(X_i+\epsilon j, X_i/\epsilon)
	.
\]
Hence express
\begin{align*}
\left<v(X_i, X_i/\epsilon)\right>_X
=~&
\left<v(X_i, X_i/\epsilon) -v \right>_{XY}
\\ =~&
\frac{1}{N p}\sum\limits_{i=1}^N \sum\limits_{j=-\lceil p/2\rceil+1}^{\lfloor p/2\rfloor}
	\left(v(X_i, X_i/\epsilon) - v(X_i+\epsilon j, X_i/\epsilon)\right)
\end{align*}
and estimate the terms in the parenthesis (for $-\lceil p/2\rceil+1\le j\le \lfloor p/2\rfloor$):
\[
\frac{v(X_i, X_i/\epsilon) - v(X_i+\epsilon j, X_i/\epsilon)}{\epsilon}
=
\left\{
\begin{array}{ll}
- \sum\limits_{k=0}^{j-1} D_X v(X_i+\epsilon k, X_i/\epsilon)
&~ j>0
\\
\sum\limits_{k=j}^{-1} D_X v(X_i+\epsilon k, X_i/\epsilon)
&~ j<0
\\
0 &~ j=0,
\end{array}
\right.
\]
\[
|v(X_i, X_i/\epsilon) - v(X_i+\epsilon j, X_i/\epsilon)|
\le
\epsilon\sum\limits_{k=-\lceil p/2\rceil+1}^{\lfloor p/2\rfloor} |D_X v(X_i + k\epsilon, X_i/\epsilon)|.
\]
Thus,
\begin{align*}
\left|\left<v(X_i, X_i/\epsilon)\right>_X\right|
\le ~&
\epsilon\,\frac{1}{N p}\sum\limits_{i=1}^N \sum\limits_{j=-\lceil p/2\rceil+1}^{\lfloor p/2\rfloor}
	\sum\limits_{k=-\lceil p/2\rceil+1}^{\lfloor p/2\rfloor} |D_X v(X_i + k\epsilon, X_i/\epsilon)|
\\ \le~&
\epsilon\,\frac{1}{N p} \sum\limits_{i=1}^N \frac{p}{2}
	\sum\limits_{k=-\lceil p/2\rceil+1}^{\lfloor p/2\rfloor} |D_X v(X_i + k\epsilon, X_i/\epsilon)|
\\ =~&
\frac{\epsilon p}{2}\, \frac{1}{N p} \sum\limits_{i=1}^N
	\sum\limits_{k=-\lceil p/2\rceil+1}^{\lfloor p/2\rfloor} |D_X v(X_i, X_i/\epsilon - k)|
=
\frac{\epsilon p}{2}\, \left<|D_X v|\right>_{XY},
\end{align*}
substituting which into the definition of $u^{\rm c}$ finishes the proof of \eqref{eq:aposteriory:uc-average}:
\[
\left|\left<u^{\rm c}\right>_X\right| \le \left|\left<u^0(X_i)\right>_X\right| + \epsilon \left|\left<v(X_i, X_i/\epsilon)\right>_X\right|
\le 0 + \frac{\epsilon^2 p}{2} \left<|D_X v|\right>_{XY}.
\]

To show \eqref{eq:aposteriory:estimateH1} we first compute, using \eqref{eq:linear-nn_exact_first_integral}, \eqref{eq:linear-nn_exact:constC}, and \eqref{eq:linear-nn_exact:homogenized-tensor},
\begin{align*}
\psi^{\epsilon} D_X u^{\rm c}
=~& \psi^{\epsilon} (D_X T_Y + \epsilon^{-1} D_Y) \left(u^0 + \epsilon\chi D_X u^0\right)
\\ =~& \psi^{\epsilon} D_X u^0 + \psi^{\epsilon} D_Y \chi D_X u^0 + \epsilon \psi^{\epsilon} D_X T_Y (\chi D_X u^0)
\\ =~&
\psi^{\epsilon} (1 + D_Y \chi) D_X u^0
+ \epsilon \psi^{\epsilon} D_X T_Y (\chi D_X u^0)
\\ =~&
\psi^0 D_X u^0
+ \epsilon \psi^{\epsilon} D_X T_Y (\chi D_X u^0)
,
\end{align*}

Hence compute, using \eqref{eq:variational_equation_linear-nns}, \eqref{eq:linear-homogenized-nn}, and the fact that $D_X \left<u^{\rm c}\right>_X = 0$,
\[
-D_X (\psi^{\epsilon} D_X (u^{\rm c}-\left<u^{\rm c}\right>_X - u))
=
- \epsilon D_X \left(\psi^{\epsilon} D_X T_Y (\chi D_X u^0)\right)
.
\]
Treating this as an equation for $(u^{\rm c}-\left<u^{\rm c}\right>_X - u) \in U_{\#}^{N}(\eps\mathbb{Z})$, upon using proposition \ref{prop:original-problem} one gets
\[
|u^{\rm c}-u|_{H^1}
\le \frac{1}{c_{\psi}} |\epsilon D_X \left(\psi^{\epsilon} D_X T_Y (\chi D_X u^0)\right)|_{H^{-1}}
\le \frac{\epsilon C_{\psi}}{c_{\psi}} \|D_X T_Y (\chi D_X u^0)\|_{L^2}.
\]
\end{proof}

\section{Homogenized QC for Complex Lattices}\label{sec:HQC}

We formulate the homogenized quasicontinuum method (HQC) --- the QC method for complex crystalline materials --- in a framework of homogenization.
We introduce HQC in the general 1D periodic case, i.e., with finite range nonlinear interaction (see \eqref{eq:variational_problem_nonlinear}),
and possibly oscillating external force $f=f^{\epsilon}$.
For the case of materials with known periodic structure (i.e., crystalline materials), the HQC method will be equivalent to applying QC to the homogenized equations.
We mention however two advantages of HQC.
First, the method is based on the original equation describing the spatially oscillating material and does not rely on effective (homogenized) equation derived beforehand.
Such strategies have proved successful in the continuum elasticity where many macro to micro methods averaging the effective equations on the fly have been derived (see the review paper \cite{GeersKouznetsovaBrekelmans2010} and the references therein).
Second, we think that the HQC can be applied to non-crystalline materials and to time-dependent zero- or even finite-temperature problems.

\subsection{HQC Method}
Consider the problem of finding an equilibrium of an atomistic material in the general 1D periodic case (i.e., with finite range nonlinear interaction) \eqref{eq:variational_problem_nonlinear}.
The method will be presented using macro-to-micro framework as used in some numerical homogenization procedures 
\cite{Abdulle2009, EEL2007, MieheBayreuther2007, GeersKouznetsovaBrekelmans2010, TeradaKikuchi2001}.

\subsubsection{Macroscopic affine deformation}
Let
\[
{\cal X}:=\{X_i=i\eps,~ i=1,\ldots N\},
\quad N\in\mathbb N,
\]
be the reference lattice for the problem \eqref{eq:variational_problem_nonlinear}.
In the set of indices $1\leq i\leq N$ we choose $K$ values ($K<N$) $i_1<\ldots<i_K$ and compose a macroscopic lattice
\[
{\cal X}_H:=\{X_{i_k};~ i=1,\ldots K\},
\]
defining the macroscopic partition of the interval
\[
{\cal T}=\{S_k=[X_{i_k},X_{i_{k+1}});~k=1,\ldots K\}
.
\]
Here we fix $i_1=1$ for convenience (we can do so without loss of generality due to translation invariance) and define $i_{K+1}=N+1$ in accordance with the periodic extension, define $H_k=\eps(i_{k+1}-i_k)$ (for $k=1,\ldots,K$) the length of the $S_k$, and $H=\max_k H_k$.
We define the space of piecewise affine (discrete) deformations by
\begin{align*}
U_\per^H=\Big\{
~&
	u^H\in U_\per^{N}(\eps\mathbb Z):
\\ ~&
u^H(X_i)|_{S_k}=\frac{X_{i_{k+1}}-X_i}{X_{{i_k+1}}-X_{i_k}}u^H(X_{i_k})+\frac{X_i-X_{i_k}}{X_{i_{k+1}}-X}u^H(X_{i_{k+1}}),~k=1\ldots,K\Big\},
\end{align*}
and
\[
U_{\#}^H
=
\left\{u^H\in U_\per^H:~\left<u^H\right>_X=0\right\}.
\]

\subsubsection{Sampling Domains}
Inside each macroscopic interval $[X_{i_k},X_{i_{k+1}})$ we choose a representative position $X_{i_k}^{\rm rep}$ and a sampling domain
\[
S_{k}^{\rm rep}= \left\{X_i:X_{i_k}^{\rm rep}\le X_i<X_{i_k}^{\rm rep}+p\eps \right\},~\quad
{\mathcal I}^{\rm rep}_k= \left\{i\in\mathbb N:X_{i_k}^{\rm rep}\le X_i<X_{i_k}^{\rm rep}+p\eps \right\},
\]
and define the operator of averaging over the sampling domain
\[
\left<w\right>_{X_i\in S_{k}^{\rm rep}} = \frac{1}{p} \sum\limits_{X_i\in S_{k}^{\rm rep}} w(X_i).
\]

The sampling domain should be chosen closer to the center of the interval $\frac{X_{i_k} + X_{i_{k+1}}}{2}$ if the material's properties vary within the interval (more precisely, if the interaction potentials for different groups of $p$ adjacent atoms are different), as we will see in Theorem \ref{thm:modeling_error}.
More sampling domains per macro interval may be considered for higher-order macro element space $U^H$.

\subsubsection{Energy and Macro Nonlinear Form}

Define the energy of the HQC method
\[
E^{\rm HQC}(u^H)
=
\sum\limits_{S_k\in{\mathcal T}} H_k \sum_{r=1}^{R}
\left<\Phi^\epsilon_r(D_r \calR_k(u^H))\right>_{X_i\in S_{k}^{\rm rep}},
\]
where
$\calR_k\left(u^H\right)$, defined by \eqref{eq:HQC:microproblem}, is the microfunction constrained by $u^H$ in the sampling domain $S_{k}^{\rm rep}$,
and $\Phi^\epsilon_{r}(z)(X_i) = \varphi^\epsilon_r(r + r z(X_i))$ is the energy of interaction of atoms $i$ and $i+r$ (i.e., $\varphi_{i,i+r}(z)$ in the notations of Section \ref{sec:problem_formulation}, cf.\ \eqref{eq:Phi}).

The functional derivative of the above energy reads
\begin{equation}
(E^{\rm HQC})'(u^H; v^H)
=
\sum\limits_{S_k\in{\mathcal T}} H_k \sum_{r=1}^{R}
\left<(\Phi^\epsilon_r)'(D_r \calR_k(u^H)), D_r \calR_k'(u^H; v^H)\right>_{X_i\in S_{k}^{\rm rep}},
\label{eq:HQC:bilinear-form}
\end{equation}
where $\calR_k'(u^H; v^H)$ is the derivative of the reconstruction $\calR_k(u^H)$,
and $(\Phi^\epsilon_r)'(z) = r \frac{\partial}{\partial z}\varphi^\epsilon_r(r + r z)$ is defined in accordance with \eqref{eq:Psi}.

\subsubsection{Microproblem}\label{HQC:nonlinear:cell}
Given a function $u^H\in U_\per^H,$ $\calR_k\left(u^H\right)$
is a function defined on $S_{k}^{\rm rep}$
such that $\calR_k\left(u^H\right)-u^H\in U_{\#}^{p}(\eps\mathbb Z)$
and
\begin{equation}
\sum_{r=1}^{R}
\left<(\Phi^\epsilon_r)'\left(D_{r} \calR_{k}\left(u^H\right)\right),~ D_{r}s\right>_{X_i\in S_{k}^{\rm rep}}
=0
\quad \forall {s}\in U_{\#}^{p}(\eps\mathbb Z)
.
\label{eq:HQC:microproblem}
\end{equation}
\begin{remark}\label{rem:nonlinear_reconstruction_stable_equilibrium}
	When modeling essentially nonlinear phenomena (e.g., martensite-austenite phase transformation), one should require that the microstructure corresponds to a stable equilibrium.
	That is, one should require, in addition to \eqref{eq:HQC:microproblem}, that $w=\calR_{k}\left(u^H\right)-u^H\in U_{\#}^{p}(\eps\mathbb Z)$ is a local minimum of $\sum_{r=1}^{R} \left<\Phi^\epsilon_{r}\left(D_{r} (u^H+w)\right)\right>_{X_i\in S_{k}^{\rm rep}}$ \cite[p.\ 238]{TadmorSmithBernsteinEtAl1999}
\end{remark}

\begin{remark}\label{rem:linear_reconstruction}
In the case of linear interaction, the reconstruction $\calR_k$ is a linear function and hence $\calR_{k}'\left(u^H; {v}^H\right) = \calR_k(v^H)$, which makes the derivative of the HQC energy \eqref{eq:HQC:bilinear-form} take the form
\begin{equation}
(E^{\rm HQC})'(u^H; v^H)
=
\sum\limits_{S_k\in{\mathcal T}} H_k \sum_{r=1}^{R}
\left<(\Phi^\epsilon_r)'(D_r \calR_k(u^H)), D_r \calR_k(v^H)\right>_{X_i\in S_{k}^{\rm rep}}.
\label{eq:HQC:bilinear-form-linear}
\end{equation}
\end{remark}

\begin{remark}\label{rem:simplify-the-bilinear-form}
The functional derivative of the HQC energy \eqref{eq:HQC:bilinear-form} can equivalently be written as
\begin{equation}
(E^{\rm HQC})'(u^H; v^H)
=
\sum\limits_{S_k\in{\mathcal T}} H_k \sum_{r=1}^{R}
\left<(\Phi^\epsilon_r)'(D_r \calR_k(u^H)), D_r v^H\right>_{X_i\in S_{k}^{\rm rep}},
\label{eq:HQC:bilinear-form_2}
\end{equation}
by noting that $D_{r} \calR_{k}'\left(u^H; {v}^H\right)=D_{r} {v}^H+\left(D_{r} \calR_{k}'\left(u^H; {v}^H\right)-D_{r} {v}^H\right)$
and that
\[
\sum_{r=1}^{R}\left<(\Phi^\epsilon_r)'\left(D_{r} \calR_{k}\left(u^H\right)\right),~ \left(D_{r} \calR_{k}'\left(u^H; {v}^H\right)-D_{r} {v}^H\right)\right>_{X_i\in S_{k}^{\rm rep}}
=0,
\]
in view of \eqref{eq:HQC:microproblem}.
Here we used the fact that $\calR_{k}'(u^H; v^H)-v^H \in U_{\#}^{p}(\eps\mathbb Z)$ which follows from taking the derivative of $\calR_{k}(u^H)-u^H \in U_{\#}^{p}(\eps\mathbb Z)$.
\end{remark}

\subsubsection{Reconstruction}\label{HQC:nonlinear:reconstruction}
The function $\calR_k\left(u^H\right)$ is defined on $S_{k}^{\rm rep}$.
We define an extension of $\calR_k\left(u^H\right)$
on $S_{k}$ by periodic extension outside $S_k^{\rm rep}$:
\begin{equation}
u^{H,c}(X_i)= u^H + \calR_k(u^H)(\overline{X}_i),
\label{eq:HQC:periodic-extension}
\end{equation}
where $\overline{X}_i = X_{i_k}^{\rm rep}+ \epsilon \Big(\frac{X_i-X_{i_k}^{\rm rep}}{\eps} {\rm~~mod~} p\Big)$, and $(\bullet {\rm~~mod~} p)$ is an integer value modulo $p$.
By extending the function $\calR_k\left(u^H\right)$ in each $S_k\in\cal T$, one obtains a function defined on $\cal X$ which we denote by $u^{H,c}$.

\subsubsection{Variational Problem}
We define the homogenized quasicontinuum approximation as the solution $u^H\in U_{\#}^H$ of
\begin{equation}
(E^{\rm HQC})'(u^H; v^H)
=
F\left({v}^H\right),\quad\forall v^H\in U_{\#}^H
\label{eq:HQC:problem}
\end{equation}
where
\begin{equation}
F\left({v}^H\right) =
\sum\limits_{S_k\in{\mathcal T}} H_k \left<f^{\epsilon}, {v}^H\right>_{X_i\in S_{k}^{\rm rep}}
.
\label{eq:HQC:RHS}
\end{equation}
If the external force is not oscillating, it could instead be computed for a single representative atom.
Well-posedness of \eqref{eq:HQC:problem} will be discussed in Section \ref{sec:HQC-convergence} for the nearest neighbor linear interaction.

\subsection{HQC Algorithm}\label{sec:HQC:algorithm}
The problem \eqref{eq:HQC:problem} is nonlinear and its practical implementation is usually done by the Newton's method.
We briefly sketch below an algorithm for solving \eqref{eq:HQC:problem}.

\subsubsection{Second Derivative of the Energy}

For the Newton's method we need to compute the second derivative of the energy (from \eqref{eq:HQC:bilinear-form_2}):
\[
\begin{array}{l} \displaystyle
(E^{\rm HQC})''(u^{H}; w^H, v^H)
\\ \displaystyle \qquad
=
\sum\limits_{S_k\in{\mathcal T}} H_k \sum_{r=1}^{R}
\left<(\Phi_r^{\epsilon})''(D_r \calR_k(u^H)) \,D_r \calR_k'(u^H; w^H),\, D_r v^H\right>_{X_i\in S_{k}^{\rm rep}},
\end{array}
\]
which, applying the similar arguments as in Remark \ref{rem:simplify-the-bilinear-form}, can be written in a symmetric form
\begin{equation}
\begin{array}{l} \displaystyle
(E^{\rm HQC})''(u^{H}; w^H, v^H)
\\ \displaystyle \qquad
=
\sum\limits_{S_k\in{\mathcal T}} H_k \sum_{r=1}^{R}
\left<(\Phi_r^{\epsilon})''(D_r \calR_k(u^H)) \,D_r \calR_k'(u^H; w^H),\, D_r \calR_k'(u^H; v^H)\right>_{X_i\in S_{k}^{\rm rep}}.
\end{array}
\label{eq:HQC:second-variation}
\end{equation}

\subsubsection{Newton's Iterations for the Macroproblem}

The algorithm based on the Newton's method consist in choosing the initial guess $u^{(0),H}\in U^H_{\#}$ and performing iterations
\begin{equation}
	(E^{\rm HQC})''(u^{(n),H};\, u^{(n+1),H} - u^{(n),H}, v^H) = F(v^H)
	\quad \forall v^H\in U_{\#}^H,
\label{eq:HQC:Newton-macro-iterations}
\end{equation}
until $u^{(n+1),H}$ becomes close to $u^{(n),H}$ in a chosen norm.

To solve the linear system \eqref{eq:HQC:Newton-macro-iterations} for $u^{(n+1),H} - u^{(n),H}\in U_{\#}^H$, we choose a nodal basis $w_k^H$ ($1\le k\le K$) of $U_\per^H$.
One way to satisfy the condition $\left<u^H\right>_X=0$ would be to perform all the computations with one basis function eliminated (e.g., to consider $w_k^H$ for $2\le k\le K$), and post-process the final solution as $u^H - \left<u^H\right>_X$.

The stiffness matrix of the system \eqref{eq:HQC:Newton-macro-iterations} will thus be
\[
A_{lm}
=
(E^{\rm HQC})''(u^{(n),H};\, w_l^{H}, w_m^{H})
\]
and the load vector will be
\[
b_m =F(w_m^{H}).
\]
As given by the formula \eqref{eq:HQC:second-variation} we need to compute the solution of microproblem $\calR_k(u^{(n),H})$ on each sampling domain $S_k^{\rm rep}$ as well as its derivative $\calR_k'(u^{(n),H}; w^H_l)$.

\subsubsection{Solution of the Microproblem}

The microproblem \eqref{eq:HQC:microproblem} can also be solved with Newton's method.
For that, one needs to choose an initial guess $u^{(n,0)}$ to $\calR_k(u^{(n),H})$, for instance $u^{(n,0)}(X_i) := u^{(n),H}(X_i)$ and solve
\[
\begin{array}{r} \displaystyle
\sum_{r=1}^{R}
D_{r} \left((\Phi_{r}^{\epsilon})'\left(D_{r} u^{(n,\nu)}\right) + (\Phi_{r}^{\epsilon})''\left(D_{r} u^{(n,\nu)}\right)
	D_{r} \left(u^{(n,\nu+1)}-u^{(n,\nu)}\right)\right)(X_i) = 0 \qquad
\\ \displaystyle
\forall {X_i\in S_{k}^{\rm rep}}
,
\end{array}
\]
with respect to $u^{(n,\nu+1)}$ constrained by $u^{(n,\nu+1)} - u^{(n),H}\in U_{\#}^{p}(\eps\mathbb Z)$, until the difference between $u^{(n,\nu+1)}$ and $u^{(n,\nu)}$ is small in a chosen norm.

After that, we can compute $w_{k,l} = \calR_k'(u^{(n),H}; w^H_l)$ by solving
\begin{equation}
\sum_{r=1}^{R}
D_{r} \left((\Phi_{r}^{\epsilon})''\left(D_{r} u^{(n,\nu)}\right) D_{r} w_{k,l}\right)(X_i) = 0,
\quad \forall {X_i\in S_{k}^{\rm rep}}
\label{eq:HQC:equation-for-reconstruction-variation}
\end{equation}
constrained by $w_{k,l} - w^H_l\in U_{\#}^{p}(\eps\mathbb Z)$.
Notice that the derivative of the basis functions $D_{r} w^H_l$ on the interval $S_k$ can either be $0$ (in which case $w_{k,l}$ equals zero identically), or $\pm \frac{1}{H_k}$.
It implies that we essentially need to solve the problem \eqref{eq:HQC:equation-for-reconstruction-variation} limited number of times (once in the 1D case, or between $d$ and $d+1$ in $\bbR^d$, depending on implementation).

Also observe that when computing $\calR_k'(u^{(n),H}; w^H_l)$, we need to invert the same linear operator
$
\sum\limits_{r=1}^{R}
D_{r} \left((\Phi_{r}^{\epsilon})''\left(D_{r} u^{(n,\nu)}\right) D_{r}\,\bullet\right)
$
as in the final Newton's iteration, which allows for some additional optimization.

\subsection{Possible Modifications of the Algorithm}\label{sec:HQC:algorithm:modifications}

First, notice that when solving for $u^{(n+1),H}$ we could linearize the problem on the previous iteration $u^{(n),H}$.
In that case we would have only the linear cell problems and thus we would need only outer Newton's iteration, but it would be required to keep the values of the micro-solution $\calR_k(u^{(n),H})$ from the previous iteration.
Moreover, even in a practical implementation of the above algorithm it may be required to keep the values of the micro-solution: one needs these values to initialize the inner Newton iterations \cite{TadmorSmithBernsteinEtAl1999}.

Another modification could be to compute the contribution of the external force $f^{\epsilon}$ in \eqref{eq:HQC:RHS} for a single atom in the case of no oscillations in $f^{\epsilon}$.

In the case of linear interaction, the algorithm becomes simpler: one does not need to do Newton iterations.
Nevertheless, even if the algorithm in subsection \ref{sec:HQC:algorithm} is applied to the linear problem, the Newton's method would converge in just one iteration.

\section{Convergence of HQC} \label{sec:HQC-convergence}

In this section we study convergence of the HQC method introduced in \eqref{eq:HQC:problem}.
We analyze the method for linear problems and nearest neighbor interaction
\eqref{eq:variational_problem_linear-nn}.
We treat the external force $f(X_i)$ in an exact manner.
We furthermore make a slight modification to the HQC method: we assume \eqref{eq:psi-multiscale-assumption} and collocate the tensor $\psi$ in \eqref{eq:HQC:bilinear-form-linear} and \eqref{eq:HQC:microproblem} in the slow variable at $X_{i^{\rm coll}_k}$ in each sampling domain $S^{\rm rep}_k$.
That is, we solve
\begin{equation}
\label{eq:HQC:problem_analyzed}
(E^{\rm HQC})'(u^H; v^H)
:=
\sum\limits_{S_k\in{\mathcal T}} H_k
\left<\psi^{\epsilon}_{\rm coll} D_X \calR_k(u^H), D_X \calR_k(v^H)\right>_{X_i\in S_{k}^{\rm rep}}
~= \left<f, v^H\right>_X
,
\end{equation}
where
\begin{equation}
\label{equ:hom_tens_collocate}
\psi^\eps_{\rm coll}(X_i) := \psi(X_{i^{\rm coll}_k}, X_i/\eps)
\quad\forall X_i\in S^{\rm rep}_k
,
\qquad\textnormal{and}
\end{equation}
\begin{equation}
\label{eq:hom_microproblem_collocate}
\left<\psi^\eps_{\rm coll} D_X \calR_{k}\left(u^H\right),~ D_X s\right>_{X_i\in S_{k}^{\rm rep}}
=0
\quad \forall {s}\in U_{\#}^{p}(\eps\mathbb Z)\
.
\end{equation}
For an illustration of a collocation point $X_{i^{\rm coll}_k}$ and a sampling domain $S_{k}^{\rm rep}$ refer to Figure \ref{fig:rep-coll-illustration}.

\begin{figure}
\begin{center}
	\includegraphics{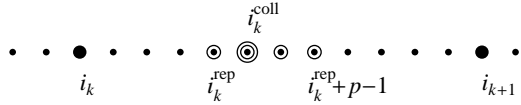}
\end{center}
\caption{Illustration of an element $S_k = [X_{i_k}, X_{i_{k+1}})$ ($i_k$ and $i_{k+1}$ are indices of nodal atoms), sampling domain with atom indices $i_k^{\rm rep}$ through $i_k^{\rm rep}+p-1$ (circled atoms) and a collocation atom $i_k^{\rm coll}$ (double-circled).}
\label{fig:rep-coll-illustration}
\end{figure}

In order to obtain the second order convergence in the $L^2$-norm, we will assume two additional conditions:
\begin{equation}
\left|X_{i^{\rm coll}_k} - \frac{X_{i_k}+X_{i_{k+1}-1}}{2}\right| \le \epsilon C_{\rm coll}
\quad
\forall S_k\in{\mathcal T}
,\qquad \textnormal{and}
\label{eq:additional_condition_1}
\end{equation}
\begin{equation}
\|D^2 \psi^0\|_{L^\infty} \le C_{\psi^0}''.
\label{eq:additional_condition_2}
\end{equation}
The condition \eqref{eq:additional_condition_1} states that the collocation point is at most $O(\epsilon)$ away from the center of each $S_k$.

In this section, the constants $C_4$, $C_5$, etc., denote generic constants which may depend on $c_{\psi}$, $C_{\psi}$, $C_{\psi}'$, $C_{\rm coll}$, $C_{\psi^0}''$, and $p$, but are independent of $\epsilon$ and $H$.

\subsection{Main Results}
\label{sub:HQC_analysis_prel}

Before proceeding with analysis (Section \ref{sub:HQC_analysis}), we summarize the main convergence results.
The results are formulated in terms of $e_{\rm mod}$, the so-called modeling error, which is defined in \eqref{eq:emod}.

\begin{theorem}
\label{thm:num_hom}
\label{lem:analysis-of-methods:homogenized-convergence}
Let $u^0,u^H$ be the solutions of problems \eqref{eq:linear-nn_homogenized} and \eqref{eq:HQC:problem_analyzed}, respectively, and assume that \eqref{equ:coerc_bound_psi} and \eqref{equ:bound_dpsi} hold.
Then there exist constants $C_4$ and $C_5$ such that
\begin{equation}
\left|u^H-u^0\right|_{H^1}\le C_4 H \|{f}\|_{L^2}
,
\qquad\textnormal{and}
\label{eq:analysis-of-methods:homogenized-convergence-H1}
\end{equation}
\begin{equation}
\left\|u^H-u^0\right\|_{L^2}\le C_{5} H^2 \|{f}\|_{L^2} + \|e_{\rm mod}\|_{L^2}
.
\label{eq:analysis-of-methods:homogenized-convergence-L2}
\end{equation}
\end{theorem}
\begin{theorem}
\label{thm:analysis-of-methods:homogenized-L2-convergence}
Under the assumptions of Theorem \ref{thm:num_hom}, let $u$ be the solution of the original problem \eqref{eq:variational_problem_linear-nn}.
Then there exist constants $C_6$ and $C_7$ such that
\[
\left\|u^H-u\right\|_{L^2}\le \left(C_6 H^2 + C_7 \epsilon\right) \|{f}\|_{L^2} + \|e_{\rm mod}\|_{L^2}.
\]
\end{theorem}
\begin{theorem}
\label{thm:analysis-of-methods:reconstructed-convergence}
Under the assumptions of Theorem \ref{thm:num_hom}, let $u$ be the solution of the original problem \eqref{eq:variational_problem_linear-nn}
and $u^{H,c}$ be the reconstruction \eqref{eq:HQC:periodic-extension} of the solution $u^H$ of \eqref{eq:HQC:problem_analyzed}, where the reconstruction operator $\calR_k(u^H)$ is defined by \eqref{eq:hom_microproblem_collocate}.
Then there exist constants $C_8$ and $C_9$ such that
\begin{equation}
\left|u^{H,c}-u\right|_{H^1}\le C_8 H \|{f}\|_{L^2}
,
\qquad\textnormal{and}
\label{eq:analysis-of-methods:reconstructed-convergence-H1}
\end{equation}
\begin{equation}
\left\|u^{H,c}-u\right\|_{L^2}\le C_9 H^2 \|{f}\|_{L^2} + \|e_{\rm mod}\|_{L^2}
.
\label{eq:analysis-of-methods:reconstructed-convergence-L2}
\end{equation}
\end{theorem}

The modeling error $e_{\rm mod}$, defined in \eqref{eq:emod}, reflects the fact that we introduce some error when neglecting the values of the tensor $\psi^0(X_i)$ everywhere outside the sampling domains $S_k^{\rm rep}$.
The modeling error is estimated in the following theorem.
\begin{theorem}\label{thm:modeling_error}
Under the assumptions of Theorem \ref{thm:num_hom},
\begin{itemize}
\item[(a)] There exists a constant $C_{10}$ such that
\begin{equation}
\|e_{\rm mod}\|_{L^2} \le \frac{1}{2 \sqrt{3}}\,|e_{\rm mod}|_{H^1} \le C_{10} H \|f\|_{H^{-1}}.
\label{eq:emod-H1-estimate}
\end{equation}
\item[(b)] If additionally \eqref{eq:additional_condition_1} and \eqref{eq:additional_condition_2} hold then there exist constants $C_{11}$ and $C_{12}$ such that
\[
\|e_{\rm mod}\|_{L^2} \le \frac{1}{2 \sqrt{3}}\,|e_{\rm mod}|_{H^1} \le (C_{11} H^2+C_{12} \epsilon) \|f\|_{H^{-1}}.
\]
\item[(c)] If the tensor $\psi(X_i, Y_i)$ in \eqref{equ:hom_tens_collocate} does not depend on $X_i$ then
$
e_{\rm mod} = 0.
$
\end{itemize}
\end{theorem}

\subsection{Error analysis}
\label{sub:HQC_analysis}

We start our analysis with the following lemma which asserts that the results stated in Section \ref{sub:HQC_analysis_prel} can be reformulated in terms of the standard QC method applied to the homogenized equations \eqref{eq:linear-nn_homogenized}.
Recall the definition of the homogenized tensor
\[
\psi^0(X_i) = \left<\psi(\bullet) \left(1+D_Y \chi(X_i, \bullet)\right)\right>_Y
\]
and define the collocated homogenized tensor $\psi^0_{\rm coll}(X_i) = \psi^0(X_{i^{\rm coll}_k})$ for $X_i\in S_k^{\rm rep}$.

\begin{lemma}
\label{lem:equv_hmm_QC_corr}
Under the assumptions of Theorem \ref{thm:analysis-of-methods:reconstructed-convergence}, the reconstruction $u^{H,c}$ \eqref{eq:HQC:periodic-extension}
can be written as
\begin{equation}
\label{eq:reconstruction_equiv}
u^{H,c}|_{S_k}= u^H + \epsilon \chi(X_{i_k^{\rm coll}}, X_i/\eps) D u^H.
\end{equation}
\end{lemma}
\begin{proof}
Fix the element $S_k$ and notice that the reconstruction defined by \eqref{eq:hom_microproblem_collocate} can be written as $\calR_{k}\left(u^H\right) = u^H + w$, where $w$ satisfies
\[
\left<\psi^\eps_{\rm coll} D_X w,~ D_X s\right>_{X_i\in S_{k}^{\rm rep}}
= - D_X u^H \left<\psi^\eps_{\rm coll},~ D_X s\right>_{X_i\in S_{k}^{\rm rep}}
\quad \forall {s}\in U_{\#}^{p}(\eps\mathbb Z)
.
\]
Notice that $\psi^\eps_{\rm coll}$ and $D_X u^H$ are constant inside each sampling domain $S_k^{\rm rep}$.
Upon substitution $w(X_i) = \tilde{w}(X_i/\epsilon) = \tilde{w}(Y_i)$ this equation takes the form
\[
\epsilon^{-1} \left<\psi^\eps_{\rm coll} D_Y \tilde{w},~ D_Y s\right>_Y
= - D_X u^H \left<\psi^\eps_{\rm coll},~ D_Y s\right>_Y
\quad \forall {s}\in U_{\#}^{p}(\mathbb Z)
,
\]
and its solution can be written as $w(X_i) = \tilde{w}(X_i/\epsilon) = \epsilon D_X u^H \chi(X_{i^{\rm coll}_k}, X_i/\eps)$, cf.\ \eqref{eq:linear-nn_chi}.
Finally, noticing that periodically extending $w$ in
\begin{equation}\label{eq:reconstruction_equiv:intermediate}
\calR_{k}\left(u^H\right)
= u^H + w
= u^H + \epsilon D_X u^H \chi(X_{i^{\rm coll}_k}, X_i/\eps)
\end{equation}
yields exactly $u^{H,c}$ concludes the proof of \eqref{eq:reconstruction_equiv}.
\end{proof}

\begin{lemma}
\label{lem:equv_hmm_QC}
Under the assumptions of Theorem \ref{thm:num_hom}, the problem \eqref{eq:HQC:problem_analyzed} is equivalent to the following problem:
find $u^H\in U^H_{\#}$ such that
\begin{equation}
\left<\psi^0_{\rm coll} D u^H, D {v}^H\right>_X=
\left<{f}, {v}^H\right>_X
\quad \forall {v}^H\in U^H_{\#}.
\label{eq:QC-for-homogenized-collocated}
\end{equation}
\end{lemma}
\begin{proof}
We continue the argument of the previous lemma: we first fix the element $S_k$ and
differentiate \eqref{eq:reconstruction_equiv:intermediate} to get
\[
D_X \calR_{k}\left(u^H\right)
= D_X u^H + \epsilon D_X u^H D_Y \chi(X_{i^{\rm coll}_k}, X_i/\eps)
= D_X u^H \left( 1 + D_Y \chi(X_{i^{\rm coll}_k}, Y_i) \right)
.
\]
Hence we compute
\begin{align*}
\left<\psi^\eps_{\rm coll} D_X \calR_{k}\left(u^H\right)\right>_{X_i\in S_{k}^{\rm rep}}
=~&
D_X u^H \left<\psi^\eps_{\rm coll} \left(1 + D_Y \chi(X_{i^{\rm coll}_k}, Y_i)\right) \right>_Y
\\ =~&
\psi^0(X_{i^{\rm coll}_k}) D_X u^H = \psi^0_{\rm coll} D_X u^H
.
\end{align*}
Finally, the following computation then shows that the left-hand side of \eqref{eq:HQC:problem_analyzed} is equal to that of \eqref{eq:QC-for-homogenized-collocated}:
\begin{align*}
(E^{\rm HQC})'(u^H; v^H)
=~&
\sum\limits_{S_k\in{\mathcal T}} H_k
\left<\psi^{\epsilon}_{\rm coll} D_X \calR_k(u^H), D_X v^H\right>_{X_i\in S_{k}^{\rm rep}}
\\ =~&
\sum\limits_{S_k\in{\mathcal T}} H_k
\left<\psi^{\epsilon}_{\rm coll} D_X \calR_k(u^H) \right>_{X_i\in S_{k}^{\rm rep}} D_X v^H
\\ =~&
\sum\limits_{S_k\in{\mathcal T}} H_k
\psi^0_{\rm coll} D_X u^H D_X v^H
=
\left<\psi^0_{\rm coll} D_X u^H D_X v^H\right>_X
,
\end{align*}
where we used Remark \ref{rem:simplify-the-bilinear-form} in the first step of this derivation, and omitted the argument $(X_i)$ or $(X_{i^{\rm coll}_k})$ of the functions $D_X u^H$, $D_X v^H$, and $\psi^0_{\rm coll}$ since they are constant on each interval $S_k\in{\cal T}$.
Thus, \eqref{eq:HQC:problem_analyzed} and \eqref{eq:QC-for-homogenized-collocated} are equivalent.
\end{proof}

Define the modeling error as
\begin{equation}
e_{\rm mod} = u^H-\tildeu^H,
\label{eq:emod}
\end{equation}
where $\tildeu^H$ is the solution of the following problem:
\begin{equation}
\left<{\psi}^0 D \tildeu^H, D {v}^H\right>_X=
\left<{f}, {v}^H\right>_X
\quad \forall {v}^H\in U^H_{\#}
.
\label{eq:QC-for-homogenized}
\end{equation}

\begin{proposition}\label{prop:discretized_problem}
Solutions $u^H$, $\tildeu^H$ of the discretized problems \eqref{eq:QC-for-homogenized-collocated} and \eqref{eq:QC-for-homogenized} exist, are unique, and satisfy the following estimates:
\begin{equation}
|u^H|_{H^1} \le c_{\psi}^{-1} |f|_{H^{-1}},
\qquad
|\tildeu^H|_{H^1} \le c_{\psi}^{-1} |f|_{H^{-1}}.
\label{eq:discretized_problem_estimates}
\end{equation}
\end{proposition}
\begin{proof}
The statement can be proved similarly to Proposition \ref{prop:original-problem}, by noticing that $c_{\psi}\le {\psi}^0\le C_{\psi}$ and $c_{\psi}\le {\psi}^0_{\rm coll} \le C_{\psi}$, and substituting the original space $U^N_{\#}(\eps\mathbb{Z})$ with the discretized space $U^H_{\#}$.
\end{proof}

We can now prove Theorem \ref{thm:modeling_error}.
\begin{proof}[Proof of Theorem \ref{thm:modeling_error}.]

{\itshape Part (c)} of the theorem is trivial: if $\psi=\psi(Y_i)$ in \eqref{equ:hom_tens_collocate}, then $\psi^0(X_i)$ is constant, hence $\psi^0_{\rm coll}$ coincides with $\psi^0$, and therefore $u^H = \tildeu^H$.

The bound in part (a) is based on the following estimate:
\begin{align}
c_\psi \left<D (\tildeu^H - u^H), D v^H\right>_X
\le~&
\left|\left<\psi^0_{\rm coll} D (\tildeu^H - u^H), D v^H\right>_X\right|
\notag \\ =~&
\left|\left<\psi^0_{\rm coll} D \tildeu^H - \psi^0_{\rm coll} D u^H, D v^H\right>_X\right|
\notag \\ =~&
\left|\left<{\psi}^0_{\rm coll} D\tildeu^H - {\psi}^0 D\tildeu^H, D {v}^H\right>_X\right|
\notag \\ \le~&
\|{\psi}^0-{\psi}^0_{\rm coll}\|_{L^{\infty}}\, c_{\psi}^{-1} \|f\|_{H^{-1}} ~|v^H|_{H^1}
,
\label{eq:modeling_error:estimate_case_a}
\end{align}
where the last estimate follows from \eqref{eq:discretized_problem_estimates}.
The difference ${\psi}^0-{\psi}^0_{\rm coll}$ in \eqref{eq:modeling_error:estimate_case_a} can be estimated as follows: for $X_i\in S_k$
\[
|\psi^0(X_i)-{\psi}^0_{\rm coll}(X_i)|
=
|\psi^0(X_i)-\psi^0(X_{i^{\rm coll}_k})|
\le
|X_{i^{\rm coll}_k} - X_i| \|D_X \psi^0\|
\le
H_k \frac{C_\psi}{c_\psi} C_\psi',
\]
hence $\|\psi^0-{\psi}^0_{\rm coll}\|_{L^{\infty}} \le H \frac{C_\psi}{c_\psi} C_\psi'$.
Taking now supremum over $|v^H|_{H^1}=1$ concludes the proof of part (a).

To show (b), first observe that in \eqref{eq:QC-for-homogenized-collocated}, $D u^H$ and $D v^H$ are constant on any interval $S_k\in{\mathcal T}$.
Therefore $\psi^0_{\rm coll}$ can be changed to any other tensor with the same average over $S_k$.
Hence define $\check{\psi}^0(X_i) := \psi^0_{\rm coll}(X_i) + \psi^0(X_i) - \left<\psi^0\right>_{X_i\in S_k}$.
Since $\left<\psi^0_{\rm coll}\right>_{X_i\in S_k} = \left<\check{\psi}^0\right>_{X_i\in S_k}$, the solution $u^H$ of \eqref{eq:QC-for-homogenized-collocated} coincides with the solution of
\[
\left<\check{\psi}^0 D u^H, D {v}^H\right>_X=
\left<{f}, {v}^H\right>_X
\quad \forall {v}^H\in U^H_{\#}
.
\]
Hence we can use the arguments of \eqref{eq:modeling_error:estimate_case_a} to estimate:
\begin{equation}
c_\psi |u^H - \tildeu^H|_{H^1}
\le
\|{\psi}^0-\check{\psi}^0\|_{L^{\infty}}\, c_{\psi}^{-1}\|f\|_{H^{-1}},
\label{eq:modeling_error:estimate_case_b}
\end{equation}
where
\begin{align*}
& \psi^0(X_m) - \check{\psi}^0(X_m)
=
\left<\psi^0\right>_{X_i\in S_k} - \psi^0_{\rm coll}(X_m)
= \left<\psi^0\right>_{X_i\in S_k} - \psi^0(X_{i^{\rm coll}_k})
\\ & \qquad  =~
\left<\psi^0(X_i) - \psi^0(X_{i^{\rm coll}_k})\right>_{X_i\in S_k}
\\ & \qquad  =~
\left<
	(X_i - X_{i^{\rm coll}_k}) D \psi^0(X_{i^{\rm coll}_k})
	+ \epsilon \sum\limits_{j = \min(i_k^{\rm coll}, i)}^{\max(i_k^{\rm coll}, i)-1} |X_i-X_{j+1}| D^2 \psi^0(X_j)
\right>_{X_i\in S_k}
\\ & \qquad  =:~
\left<Q_1 + Q_2\right>_{X_i\in S_k}
.
\end{align*}
It is straightforward to show that $Q_1$ averages up to
\[
\left<Q_1\right>_{X_i\in S_k}
=
	\left(\frac{X_{i_{k+1}-1}+X_{i_k}}{2} - X_{i^{\rm coll}_k}\right) D \psi^0(X_{i^{\rm coll}_k}),
\]
and can be effectively estimated using \eqref{eq:additional_condition_1}.
The second terms can be estimated as
\[
|Q_2|
\le
\|D^2 \psi^0(X_j)\|_{L^\infty}\,\epsilon \sum\limits_{j = \min(i_k^{\rm coll}, i)}^{\max(i_k^{\rm coll}, i)-1} |X_i-X_{j+1}|
\le
\|D^2 \psi^0(X_j)\|_{L^\infty} \frac{1}{2} H_k^2
.
\]
Thus, combining these estimates, one gets
\[
|\psi^0(X_m) - \check{\psi}^0(X_m)| \le
\epsilon C_{\rm coll} |D \psi^0(X_{i^{\rm coll}_k})| + \frac{1}{2} H_k^2 \|D^2 \psi^0(X_j)\|_{L^\infty}.
\]
Taking maximum over all $X_m$,
\[
\|\psi^0 - \check{\psi}^0\|_{L^\infty}
\le
\epsilon C_{\rm coll}\,\frac{C_\psi}{c_\psi}C_\psi' + \frac{1}{2} H^2 C_{\psi^0}''
,
\]
and substituting $\|\psi^0 - \check{\psi}^0\|_{L^\infty}$ into \eqref{eq:modeling_error:estimate_case_b} yields the desired result.
\end{proof}

In view of Lemma \ref{lem:equv_hmm_QC}, we will turn to analysis of the problem \eqref{eq:QC-for-homogenized}.
We next introduce the (homogenized) energy norm in the space $U^N_{\#}(\eps\mathbb{Z})$:
\[
\|{w}\|_{\psi^0}^2 = \left<\psi^0 D {w}, D {w}\right>_X.
\]
Obviously, under the assumption \eqref{equ:coerc_bound_psi}, due to the estimate \eqref{equ:coerc_bound_psi0}, the energy norm is equivalent to the $H^1$-norm:
\begin{equation}
\label{equ:norm_equiv}
c_{\psi} |\bullet|_{H^1}^2 \le \|\bullet\|_{\psi^0}^2 \le C_{\psi} |\bullet|_{H^1}^2.
\end{equation}

\begin{lemma}
\label{lem:best-approximation}
Under the assumptions of Theorem \ref{thm:apriory-H1}, let $u^0\in U^N_{\#}(\eps\mathbb{Z})$ be the solution to the exact homogenized equations \eqref{eq:linear-nn_homogenized}, and $\tildeu^H\in U^H_{\#}$ be the solution to the QC equations \eqref{eq:QC-for-homogenized}.
Then $\tildeu^H$ is the best approximation to the exact solution $u^0$ in the energy norm, i.e.,
\begin{equation}
\label{equ:best-approximation}
\left\|\tildeu^H-u^0\right\|_{\psi^0} \le \left\|{v}^H-u^0\right\|_{\psi^0}
\quad
\forall v^H\in U^H_{\#}.
\end{equation}
\end{lemma}
\begin{proof}
The result follows from
\[
\left<{\psi}^0 D (\tildeu^H-u^0), D {v}^H\right>_X=0
\quad \forall {v}^H\in U^H_{\#},
\]
which states that $\tildeu^H$ is the orthogonal projection (in the energy norm) of $u^0$ onto $U^H_{\#}$.
\end{proof}

Following the standard procedure for the analysis of finite element methods (FEM) we will estimate $\left\|u^0-I_H u^0\right\|_{\psi^0}$, where $I_H u^0$ is the nodal interpolant of $u^0$.
This interpolant is defined for every function $v\in U_\per^{N}(\epsilon\mathbb Z)$
in the following way.
For a partition ${\cal X}_H$ of ${\cal X}$ define a function $\hat I_H v\in U_\per^{N}(\epsilon\mathbb Z)$ such that
\begin{equation}
\label{equ:interpolant}
\hat I_H v(X_{i_k})=v(X_{i_k}),
\quad k=1,\ldots K.
\end{equation}
Then set
\[
I_H v=\hat I_H v- \left<\hat I_H v\right>_X.
\]
Thus, for $v\in U_\per^{N}(\epsilon\mathbb Z)$ we have $I_H v \in U^H_{\#}$.

\begin{lemma}
\label{lem:interpolation-error}
The following estimate holds:
\[
\left|{v}-I_H v\right|_{H^1}
\le
\frac{1}{2\sqrt{3}} H \left|v\right|_{H^2}
\quad \forall v\in U_{\#}^{N}(\epsilon\mathbb Z)
,
\]
where $H = \max\limits_{1\le k\le K} H_k$ and $H_k = \epsilon (X_{i_{k+1}}-X_{i_k})$.
\end{lemma}
\begin{proof}
By noting that $\left|v-I_H v\right|_{H^1}^2=\left|v-\hat I_H v\right|_{H^1}^2$, the result follows from
\begin{align*}
\left|v-I_H v\right|_{H^1}^2
=~&
\epsilon \sum_{k=1}^{K} \sum\limits_{i=i_k}^{i_{k+1}-1} |D \left(v(X_i)-\hat I_H v(X_i)\right)|^2
\\ \le~&
\epsilon \sum_{k=1}^{K} \frac{H_k^2}{6} \sum\limits_{i=i_k}^{i_{k+1}-2} \left|D^2 \left(v(X_i)-\hat I_H v(X_i)\right)\right|^2
\\ \le~&
\epsilon \frac{H^2}{6} \sum_{k=1}^{K} \sum\limits_{i=i_k}^{i_{k+1}-1} \left|D^2 v(X_i)\right|^2
= \frac{H^2}{6} \left| v\right|_{H^2}^2,
\end{align*}
where we used the discrete Poincar\'e inequality \eqref{eq:discrete-poincare} in the first step (notice that $\sum\limits_{i=i_k}^{i_{k+1}-1}D \left( v(X_i)-\hat I_H v(X_i)\right) = 0$
from \eqref{equ:interpolant}), and the fact that $D^2\hat I_H v =0$ in the second step.
\end{proof}

\begin{proof}[Proof of Theorem \ref{thm:num_hom}]
The estimate \eqref{eq:analysis-of-methods:homogenized-convergence-H1} (convergence in the $H^1$-norm) follows from \eqref{equ:norm_equiv}, \eqref{equ:best-approximation}, Lemma \ref{lem:interpolation-error}, and \eqref{eq:emod-H1-estimate}.
To prove
\eqref{eq:analysis-of-methods:homogenized-convergence-L2} (convergence in the $L^2$-norm)
we use the standard duality arguments.
Consider
\[
\left<{\psi}^0 D {w}^0, D {v}\right> = \left<u^0-\tildeu^H,{v}\right>
\quad \forall {v}\in U_{\#}^{N}(\epsilon\mathbb Z),
\]
\[
\left<{\psi}^0 D {w}^H, D {v}^H\right> = \left<u^0-\tildeu^H,{v}^H\right>
\quad \forall {v}^H\in U^H_{\#}.
\]

Then
\begin{align*}
\|u^0-\tildeu^H\|_{L^2}^2
=~&
\left<u^0-\tildeu^H, u^0-\tildeu^H\right>_X
=
\left<{\psi}^0 D {w}^0, D \left(u^0-\tildeu^H\right)\right>_X
\\ =~&
\left<{\psi}^0 D \left(u^0-\tildeu^H\right), D {w}^0\right>_X
=
\left<{\psi}^0 D \left(u^0-\tildeu^H\right), D \left({w}^0-{w}^H\right)\right>
\\ \le~&
C_{\psi}\left|u^0-\tildeu^H\right|_{H^1} \left|{w}^0-{w}^H\right|_{H^1}
\le
C_{\psi} C_4 H \|{f}\|_{L^2} C_4 H \left\|u^0-\tildeu^H\right\|_{L^2},
\end{align*}
hence
\[
\|u^0-u^H\|_{L^2}
\le
\|u^0-\tildeu^H\|_{L^2}
+
\|\tildeu^H-u^H\|_{L^2}
\le
C_{\psi} C_4^2 H^2 \|{f}\|_{L^2}
+
\|e_{\rm mod}\|_{L^2}
.
\]
\end{proof}

\begin{proof}[Proof of Theorem \ref{thm:analysis-of-methods:homogenized-L2-convergence}]
Follows immediately from Theorems \ref{thm:apriory-L2} and \ref{thm:num_hom}.
\end{proof}

\begin{lemma}
\label{lem:analysis-of-methods:reconstructed-convergence-to-uc}
Under the assumptions of Theorem \ref{thm:analysis-of-methods:reconstructed-convergence}, let $u^{H,c}$ be the reconstruction
\eqref{eq:HQC:periodic-extension} and $u^c$ be the corrector defined in \eqref{equ:corrector}.
Then there exist constants $C_{13}$ and $C_{14}$ such that
\begin{align}
\label{eq:analysis-of-methods:reconstructed-convergence-to-uc-H1}
\left|\hat{u}^{H,c}-u^{H,c}\right|_{H^1} \le~& C_{13} H \|{f}\|_{H^{-1}},
\qquad\textnormal{and}
\\
\label{eq:analysis-of-methods:reconstructed-convergence-to-uc-L2}
\left\|\hat{u}^{H,c}-u^{H,c}\right\|_{L^2} \le~& C_{14}\,\eps H \|{f}\|_{H^{-1}},
\end{align}
where
\[
\hat u^{H,c}=u^H + \epsilon {\chi}(X_i, X_i/\eps) D u^H,
\]
and ${\chi}$ is defined as a solution to \eqref{equ:chi} (or \eqref{equ:chi_variat}).
\end{lemma}
\begin{proof}
Notice that
\[
\hat u^{H,c}-u^{H,c}
=
\epsilon(\chi - \chi_{\rm coll}) Du^H,
\]
where collocated $\chi$ is defined as $\chi_{\rm coll}(X_j, Y_i) = \chi(X_{i_k^{\rm coll}}, Y_i)$ for $X_j\in S_k$, and can be estimated using \eqref{eq:DXchi-estimate} as
\begin{equation}
|\chi(X_j, Y_i)-\chi_{\rm coll}(X_j, Y_i)| \le |X_{i_k^{\rm coll}}-X_j|\,\|D\chi\|_{L^{\infty}}
\le H\,p \frac{C'_{\psi}}{c_{\psi}}
.
\label{eq:analysis-of-methods:reconstructed-convergence-to-uc-L2_intermediate}
\end{equation}
Then, using \eqref{eq:analysis-of-methods:reconstructed-convergence-to-uc-L2_intermediate} and \eqref{eq:discretized_problem_estimates}
we obtain \eqref{eq:analysis-of-methods:reconstructed-convergence-to-uc-L2}:
\[
\|\hat u^{H,c}-u^{H,c}\|_{L^2} \le \epsilon \|\chi - \chi_{\rm coll}\|_{L^{\infty}} \|Du^H\|_{L^2}
\le
\epsilon\,H p \frac{C'_{\psi}}{c_{\psi}}\,c_{\psi}^{-1} \|f\|_{H^{-1}}
,
\]
from where \eqref{eq:analysis-of-methods:reconstructed-convergence-to-uc-H1} follows directly by applying the inverse discrete Poincar\'e inequality \eqref{eq:inverse-Poincare}.
\end{proof}

\begin{lemma}
\label{lem:analysis-of-methods:reconstructed-convergence-to-hat_uc}
Under the assumptions of Theorem \ref{thm:analysis-of-methods:reconstructed-convergence}, let $u^c$ be the corrector defined in \eqref{equ:corrector}.
Then there exist constants $C_{15}$ and $C_{16}$ such that
\begin{align*}
\left|\hat u^{H,c}-u^{\rm c}\right|_{H^1}\le\,& C_{15} H \|{f}\|_{L^2}
,
\qquad\textnormal{and}
\\
\left\|\hat u^{H,c}-u^{\rm c}\right\|_{L^2}\le\,& C_{16} H^2 \|{f}\|_{L^2},
\end{align*}
\end{lemma}
\begin{proof}
We express
\[
\hat u^{H,c}-u^{\rm c}
=
\big[u^H - u^0\big] +
\big[\epsilon {\chi} D \left(u^H - u^0\right)\big],
\]
and estimate the second term of the right-hand as
\begin{eqnarray*}
\left|\epsilon {\chi} D \left(u^H - u^0\right)\right|_{H^1}&\leq&
p \frac{C_{\psi}}{c_{\psi}} \left|u^H - u^0\right|_{H^1}
\\
\left|\epsilon {\chi} D \left(u^H - u^0\right)\right|_{L^2}&\leq&
p \frac{C_{\psi}}{c_{\psi}} \left\|u^H - u^0\right\|_{L^2}.
\end{eqnarray*}
using the inverse discrete Poincar\'e inequality \eqref{eq:inverse-Poincare}
and the estimate \eqref{eq:Xchi-estimate}.
The result follows then from Theorem \ref{thm:num_hom}.
\end{proof}

\begin{proof}[Proof of Theorem \ref{thm:analysis-of-methods:reconstructed-convergence}]
The inequalities \eqref{eq:analysis-of-methods:reconstructed-convergence-H1}
and \eqref{eq:analysis-of-methods:reconstructed-convergence-L2} follow from Lemmas
\ref{lem:analysis-of-methods:reconstructed-convergence-to-uc}, \ref{lem:analysis-of-methods:reconstructed-convergence-to-hat_uc}, and \ref{lem:Hminus1-through-L2}, the fact that $\epsilon H \le H^2$, and Theorem \ref{thm:modeling_error}.
\end{proof}

\section{Example of Application to a 2D Lattice}\label{sec:2d}

In this section, we consider a simple 2D model to illustrate how the proposed approach can be applied to 2D materials.

\subsection{Notations}
All the coordinates and atom indices will be vectors with two components, for instance $X_i = (X_{i,1}, X_{i,2})$, $i=(i_1, i_2)$.
The unit vectors in our 2D space will be denoted as $e_1 = (1,0)$ and $e_2 = (0,1)$.
The length of a 2D vector $v$ is denoted as $|v| = \left(v_1^2+v_2^2\right)^{1/2}$, the scalar product of two vectors $v$ and $w$ is denoted as $v\cdot w = v_1 w_1 + v_2 w_2$.

We define the inequalities for 2D vectors in the following way: $u<v$ if, by definition, $u_1<v_1$ and $u_2<v_2$ (likewise for relations $>$, $\le$, and $\ge$).
Thus, $(1,1)\le i\le N$ means $1\le i_1\le N_1$ and $1\le i_2\le N_2$.
We will also use the associated notations for the sums, for instance $\sum\limits_{i=(1,1)}^N\bullet$\ .

\subsection{Equations of Equilibrium}

Consider a square lattice with the reference configuration of atoms given by
\[
X_i = \epsilon i
\quad ((1,1)\le i\le N).
\]
The position of the atoms $x_i$ and displacements $u_i$ are related through $x_i = X_i + u_i$.
We consider the system with $N$-periodic ($N=(N_1, N_2)$) conditions $u_{i+N} = u_i$.
The space of such $N$-periodic vector-valued sequences is denoted as $U_\per^N$.

Consider the following operators on $U_\per^N$:
\[
D_{\alpha} u_i = \frac{u_{i+e_{\alpha}}-u_i}{\epsilon}
\quad(\alpha=1,2)
, \quad
D_{r} u_i = \frac{u_{i+r}-u_i}{\epsilon |r|}
,
\]
the averaging in $i$:
\[
\left<\vecu \right>_i = \frac{1}{N_1 N_2} \sum\limits_{i=(1,1)}^N u_i
\]
and the scalar product
\[
\left<\vecu ,\vec{v}\right>_i = \left<\vecu \cdot\vec{v}\right>_i = \frac{1}{N_1 N_2} \sum\limits_{i=(1,1)}^N u_i\cdot v_i.
\]

Consider the linear interaction of atoms defined by a set of neighbors ${\mathcal R}$ so that the functional derivative of the interaction energy is
\[
\Eint'(u;v) = \sum\limits_{r\in {\mathcal R}}\left<\vec{\psi}_r D_{r} \vecu , D_{r} \vec{v}\right>_i,
\]
where $\psi_{r,i}$ defines interaction between atoms $i$ and $i+r$.
Such linear interaction corresponds to a spring model with zero equilibrium length (cf.\ \cite{FrieseckeTheil2002} for the discussion on the nonlinear model with ideal springs but with nonzero spring equilibrium length).
The derivative of the external potential energy is
\[
\Eext'(\vec{v}) = -\left<\vec{f}, \vec{v}\right>_i.
\]
Thus, the equilibrium equation has the form
\begin{equation}
\sum\limits_{r\in {\mathcal R}}\left<\vec{\psi}_r D_{r} \vecu , D_{r} \vec{v}\right>_i
=
\left<\vec{f}, \vec{v}\right>_i
\quad \forall \vec{v}\in U_\per^N.
\label{eq:2d:variational}
\end{equation}

\subsection{Homogenization}

Homogenization of equations \eqref{eq:2d:variational} follows Section \ref{sec:atm-hmg}.
By analogy with the 1D case, we will use the term ``vector-valued function'' (or, in short, ``function'') for $u=u(X_i, Y_j)$ rather than the term ``vector field''.

\subsubsection{Fast and Slow Variables}
We first define the fast variable $Y_i = X_i/\epsilon$, the differentiation operators
\[
\begin{array}{c@{~}c@{~}c@{\qquad}c@{~}c@{~}c}
	D_{X,r} v(X_i,Y_j) & = & \displaystyle
	\frac{v(X_{i+r},Y_j)-v(X_i,Y_j)}{\left|X_{i+r}-X_i\right|},
	& D_{X_{\alpha}} & = & D_{X, e_{\alpha}}
\\[1em] \displaystyle
D_{Y,r} v(X_i,Y_j) & = & \displaystyle
	\frac{v(X_i,Y_{j+r})-v(X_i,Y_j)}{\left|Y_{i+r}-X_i\right|},
	& D_{Y_{\alpha}} & = & D_{Y, e_{\alpha}},
\end{array}
\]
where $r\in{\mathbb Z}^2$, $r\ne0$, $\alpha=1,2$,
the translation operators in $Y$:
\[
T_{Y,r} v(X_i,Y_j) = v(X_i,Y_{j+r}),
\quad
T_{Y_{\alpha}} v(X_i,Y_j) = v(X_i,Y_{j+e_{\alpha}}),
\]
and averaging and scalar products:
\[
\begin{array}{r@{~}c@{~}l@{\qquad}r@{~}c@{~}l}
\left<u\right>_X &=& \frac{1}{N_1 N_2} \sum\limits_{i=(1,1)}^N u(X_i, Y_j),
&
\left<u,v\right>_X &=& \left<u\cdot v\right>_X, 
\\[1em]
\left<u\right>_Y &=& \frac{1}{p_1 p_2} \sum\limits_{j=(1,1)}^p u(X_i, Y_j),
&
\left<u,v\right>_Y &=& \left<u\cdot v\right>_Y, 
\\[1em]
\left<u\right>_{XY} &=& \left<\left<u\right>_Y\right>_X,
&
\left<u,v\right>_{XY} &=& \left<u \cdot v\right>_{XY}.
\end{array}
\]
The space of $p$-periodic functions ($p=(p_1,p_2)$) is denoted as $U_\per^p$ and hence the functions of $X$ and $Y$ belong to $U_\per^N\otimes U_\per^p$.

\subsubsection{Solution Representation}
Assume, as before, $u_i = u(X_i, Y_i)$ and $\psi_{r,i} = \psi_r(X_i, Y_i)$.
Then $D_{r} = D_{X,r} T_{Y,r} + \epsilon^{-1} D_{Y,r}$ for such functions $u$ and $\psi_r$.
Same as in 1D case, use the following $O(\epsilon)$ identity:
\[
D_{r} \simeq \frac{r_1}{|r|} D_{X_1} + \frac{r_2}{|r|} D_{X_2} + \epsilon^{-1} D_{Y,r}
= \sum\limits_{\alpha=1}^2 \frac{r_{\alpha}}{|r|} D_{X_{\alpha}} + \epsilon^{-1} D_{Y,r}
.
\]
Then the equation \eqref{eq:2d:variational} takes the form
\begin{equation}
\begin{array}{r} \displaystyle
\sum\limits_{r\in {\mathcal R}}\left<\psi_r \Big(\sum\limits_{\alpha=1}^2 \frac{r_{\alpha}}{|r|} D_{X_{\alpha}} + \epsilon^{-1} D_{Y,r}\Big) u, \Big(\sum\limits_{\alpha=1}^2 \frac{r_{\alpha}}{|r|} D_{X_{\alpha}} + \epsilon^{-1} D_{Y,r}\Big) v\right>_{XY}
=
\left<f, v\right>_{XY}
\qquad \\ \displaystyle
\forall v\in U_\per^N\otimes U_\per^p.
\end{array}
\label{eq:2d:variational-fast-and-slow}
\end{equation}
We substitute $u(X_i, Y_j) = u^0(X_i, Y_j) + \epsilon u^1(X_i, Y_j) + \epsilon^2 u^2(X_i, Y_j) + O(\epsilon^3)$ into \eqref{eq:2d:variational-fast-and-slow} and collect the respective powers of $\epsilon$.
As before, by collecting the $O(\epsilon^{-2})$ and the $O(\epsilon^{-1})$ terms we obtain that $u^0(X_i,Y_j) = u^0(X_i)$ and $u^1(X_i,Y_j) = \sum\limits_{\alpha=1}^2 \chi_{\alpha}(X_i; Y_j) D_{X_{\alpha}} u^0(X_i) + \baru^1(X_i)$, where the matrix-valued functions $\chi_1$ and $\chi_2$ are defined as solutions to
\[
\sum_{r=1}^{R} \left<\psi_{r} D_{Y,r} \chi_{\alpha} e_{\beta}, D_{Y,r} s\right>_Y = -\sum_{r=1}^{R} \left<\psi_{r} \frac{r_{\alpha}}{|r|} e_{\beta}, D_{Y,r} s\right>_Y
\quad \forall s\in U_\per^p, \ \beta=1,2
.
\]
Collecting the $O(\epsilon^0)$ terms yields
\[
\sum\limits_{\alpha=1}^2 \sum\limits_{\beta=1}^2 \left<\psi^0_{\alpha\beta} D_{X,\alpha} u^0, D_{X,\beta} v\right>_X
= \sum\limits_{\beta=1}^2 \left<f, v_{\beta}\right>_X,
\]
where the homogenized tensors $\psi^0_{\alpha\beta}$ are defined as
\[
\psi^0_{\alpha\beta} = \sum_{r=1}^{R} \left<\psi \left(\frac{r_{\alpha}}{|r|} I + D_{Y,r}\right) \chi_{\alpha} e_{\beta}\right>_Y,
\]
with $I$ denoting a $2\times 2$ identity matrix.
The homogenized tensors $\psi^0_{\alpha\beta}$ are related to the fourth-order stiffness tensor in linear elasticity theory \cite{Abdulle2006, OleinikShamaevYosifian1992, Sands2002}.

\subsubsection{Example of Application of Homogenization}\label{sec:2d-example}

To illustrate how the 2D discrete homogenization works, we apply it to the following model problem.
The set of neighbors is defined by ${\mathcal R} = \left\{(1,0), (1,1), (0,1), (-1,1)\right\}$ (we omit the neighbors that can be obtained by reflection around $(0,0)$) and the interaction tensor as
\[
\psi_{(1,1),i} = \psi_{(1,-1),i} = k_3
,\quad
\psi_{(1,0),i} = \psi_{(0,1),i} = \left\{
	\begin{array}{lcl}
	k_1 & & i_1 + i_2 \textnormal{ is even} \\
	k_2 & & i_1 + i_2 \textnormal{ is odd.}
	\end{array}
\right.
\]
Such material is illustrated in Fig.\ \ref{fig:2d-springs}.

This example was motivated by the study of Friesecke and Theil \cite{FrieseckeTheil2002}, where a similar model was considered.
Friesecke and Theil considered the model with springs similar to the one illustrated in Fig.\ \ref{fig:2d-springs}, which however was nonlinear due to nonzero equilibrium distances of the springs (so that the energy of the spring between masses $x_i$ and $x_j$ is proportional to $|x_i-x_j|^2-l_0^2$, where $l_0$ is the equilibrium distance).
They found that with certain values of parameters the lattice looses stability to non-Cauchy-Born disturbances and the lattice period doubles (thus the lattice ceases to be a Bravais lattice).

The results, given with no details of actual derivation, are the following:
The period of spatial oscillations in this case is $(2,2)$.
The function $\chi$ has the form $\chi = \chi(Y_j) = (-1)^{j_1+j_2} \frac{k_1-k_2}{4 (k_1+k_2)} I$
(here $I$ is the $2\times 2$ identity matrix).
The homogenized tensors have the form
\[
\psi^0_{11} = \psi^0_{22} =
\left(k_1+k_2 + \frac{4 k_1 k_2}{k_1+k_2} + 8 k_3\right) I
,
\quad
\psi^0_{12} = \psi^0_{21} =
-\frac{(k_1-k_2)^2}{k_1+k_2} I.
\]

\subsection{HQC}

In this subsection we sketch a formulation the HQC method based on the discrete triangular elements.
Namely, we choose a partition ${\mathcal T}$ with triangles $S_k\in{\mathcal T}$ of the original atomistic domain ${\mathcal X} = \{i:\ (1,1)\le i\le N\}$.
The space of piecewise affine deformations is denoted as $U_\per^H$.

Inside each triangle $S_k$ choose a sampling rectangle $S_{k}^{\rm rep} \subset S_k$ of the size $p_1\times p_2$ atoms.
Define the HQC energy variation:
\[
(E^{\rm HQC})'(u^H, v^H)
=
\sum\limits_{S_k\in{\mathcal T}} |S_k| \sum_{r=1}^{R}
\left<\psi_r D_r \calR_k(u^H), D_r \calR_k(v^H)\right>_{X_i\in S_{k}^{\rm rep}},
\]
where $|S_k|$ is the area of the triangle, and the microfunction $\calR_k\left({w}^H\right)$ is defined on $S_{k}^{\rm rep}$ so that $\calR_k\left({w}^H\right)-u^H\in U_{\#}^{p}(\eps\mathbb Z^2)$
and
\[
\sum_{r=1}^{R}
\left<\psi_{r} D_{r} \calR_{k}\left({w}^H\right),~ D_{r}s\right>_{X_i\in S_{k}^{\rm rep}}
=0
\quad \forall {s}\in U_{\#}^{p}(\eps{\mathbb Z}^2)
.
\]
As before, the function $\calR_k\left({w}^H\right)$ can be extended on the whole triangle $S_k$ if required.

The variational problem to be solved thus becomes
\[
(E^{\rm HQC})'(u^H, v^H)
=
F\left({v}^H\right),\quad\forall v^H\in U_{\#}^H,
\]
where
\[
F\left({v}^H\right) =
\sum\limits_{S_k\in{\mathcal T}} |S_k| \left<f, {v}^H\right>_{X_i\in S_{k}^{\rm rep}}.
\]

\section{Numerical Examples}\label{sec:numeric}

We solve numerically several model problems to illustrate the performance of HQC.
We consider the linear and the nonlinear 1D model problems (Sections \ref{sec:numeric:1d-linear} and \ref{sec:numeric:1d-nonlinear}), followed by the two 2D linear model problems (Sections \ref{sec:numeric:2d-first} and \ref{sec:numeric:2d-first}).
We also study dependence of the numerical error on $p$, the spatial period of heterogeneity of the material (Section \ref{sec:numeric:different-p}).

The aim of the numerical experiments is twofold.
First, we verify numerically the sharpness of the obtained error for the 1D linear case.
Second, we investigate whether the HQC convergence results obtained for the nearest neighbor linear interaction in 1D are valid for more general cases.
The numerical results show that all theoretical conclusions made in Section \ref{sec:HQC-convergence} are also valid for finite range nonlinear interaction or for 2D problems.

\subsection{1D Linear}\label{sec:numeric:1d-linear}

In the first numerical example we solve the problem \eqref{eq:variational_problem_nonlinear} for the linear interaction case with the period of spatial oscillation $p=2$ and number of interacting neighbors $R=3$.
The potential is defined as
\[
\varphi_{i,i+r}(z) = \frac{1}{2} k_{i,i+r} 3^{1-r} (z-r)^2
\quad (1\le r\le R)
,
\]
where
\begin{equation}
k_{i,i+r} = \left\{
	\begin{array}{lcl}
	1 & & \textnormal{$i$ is even} \\
	2 & & \textnormal{$i$ is odd}
	\end{array}
\right.
\label{eq:numeric:1d-linear:k}
\end{equation}
Such potential is periodic, hence, as suggested by Theorem \ref{thm:modeling_error}, $e_{\rm mod}=0$.
The number of atoms was $N=2^{14}=16384$, and the external force was taken as
\[
f_i = \sin\left(1+2\pi X_i\right).
\]

\begin{figure}
\begin{center}
\hfill
	\includegraphics{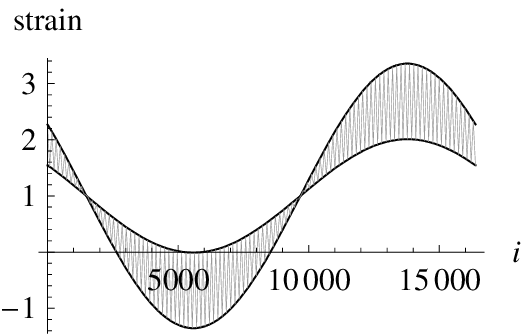}
\hfill\hfill
	\includegraphics{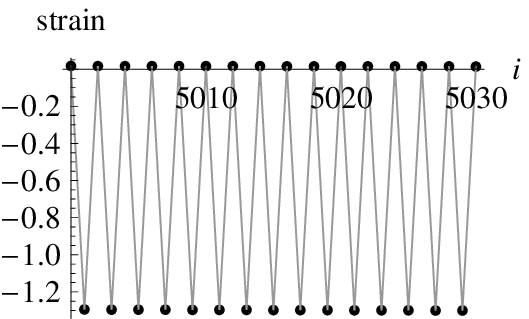}
\hfill $\mathstrut$
\end{center}
\caption{Strain $D u_i$ of the solution of the 1D linear problem: the schematically shown complete solution (left) and the closeup of the micro-structure for 31 atoms (right).}
\label{fig:solution-linear}
\end{figure}

The strain $D u_i$ for such problem is shown in Fig.\ \ref{fig:solution-linear}.

\begin{figure}
\begin{center}
\includegraphics{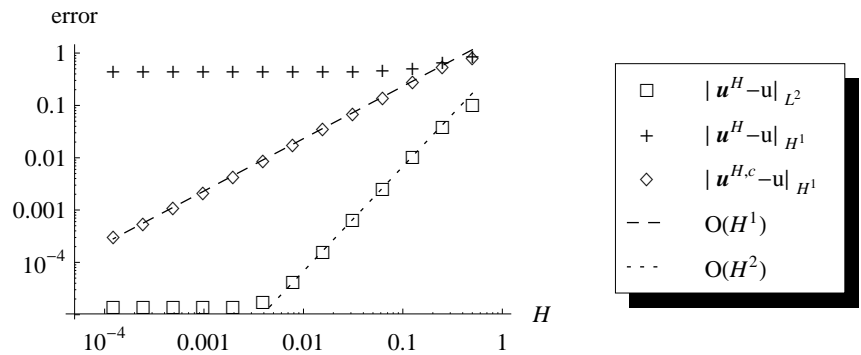}
\end{center}
\caption{Results for the 1D linear problem: errors of the post-processed HQC solution $u^{H,c}$ and the homogenized solution $u^H$ in different norms.
The errors behave in accordance with Theorems \ref{thm:analysis-of-methods:homogenized-L2-convergence} and \ref{thm:analysis-of-methods:reconstructed-convergence}.}
\label{fig:error1-linear}
\end{figure}

\begin{figure}
\begin{center}
\includegraphics{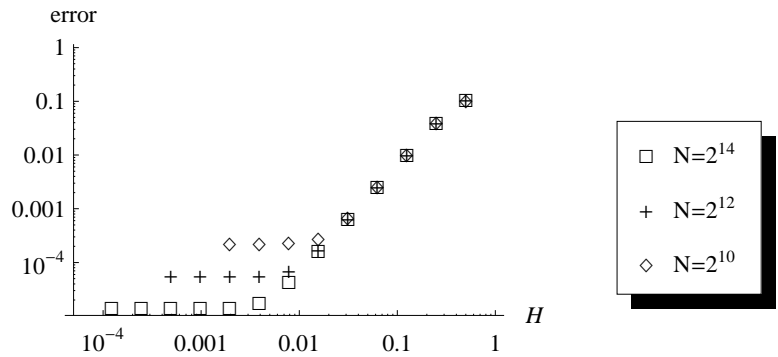}
\end{center}
\caption{Results for the 1D linear problem: errors $\|u^H-u\|_{L^2}$ for $N=2^{14}$, $N=2^{12}$, and $N=2^{10}$.
We can see that the plateau for small $H$ follows the $O(\epsilon)=O(N^{-1})$ behavior, as predicted by Theorem \ref{thm:analysis-of-methods:homogenized-L2-convergence}.}
\label{fig:error1-linear-differentN}
\end{figure}

Figure \ref{fig:error1-linear} is aimed to illustrate theorems \ref{thm:analysis-of-methods:homogenized-L2-convergence} and \ref{thm:analysis-of-methods:reconstructed-convergence}.
It can be seen that the homogenized HQC solution converges to the exact solution in the $L^2$-norm, does not converge in the $H^1$-norm, but the post-processed HQC solution does converge in the $H^1$-norm.
The convergence of the homogenized solution $u^H$ in the $L^2$-norm is exactly as suggested by Theorem \ref{thm:analysis-of-methods:homogenized-L2-convergence}: first it decreases with the second order as $H$ is refined, and later it stays constant as $H$ is refined further.
The $O(\epsilon)$ behavior of the lower bound of $\|u^H-u\|_{L^2}$ is illustrated in Fig.\ \ref{fig:error1-linear-differentN}.

\begin{figure}
\begin{center}
\includegraphics{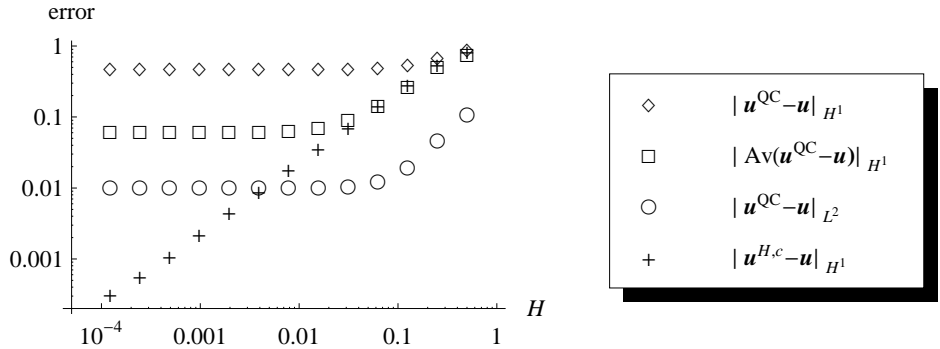}
\end{center}
\caption{Results for the 1D linear problem: errors of the post-processed HQC solution $u^{H,c}$ and solution by the naive QC method $u^{\rm QC}$.
Here $u$ is the exact solution, ${\rm Av}$ is the averaging operator defined as ${\rm Av}(u)_i = \frac{u_i+u_{i+1}}{2}$.
The graph illustrates that the naive of QC to a complex material fails, while the HQC successfully convergence to the exact atomistic solution.}
\label{fig:error2-linear}
\end{figure}

The errors of the post-processed HQC solution $u^{H,c}$ and the solution by the naive QC method $u^{\rm QC}$ are shown in fig.\ \ref{fig:error2-linear}.
It can be seen that only the solution by HQC converges to the exact solution $u$, but the solution with the naive QC method does not converge, even when compared to the averaged exact solution (averaging operator is ${\rm Av}(u)_i = \frac{u_i+u_{i+1}}{2}$) or computed in an $L^2$-norm.
These findings are similar to calculations of Tadmor et al \cite{TadmorSmithBernsteinEtAl1999} which show that assuming a linear interpolation for a silicon crystal greatly overestimates its strain energy density.

\subsection{1D Nonlinear}\label{sec:numeric:1d-nonlinear}

We solve the problem \eqref{eq:variational_problem_nonlinear} for a general nonlinear interaction case, with the period of spatial oscillation $p=2$, number of interacting neighbors $R=3$, and number of atoms $N=2^{14}=16384$.
We chose Lennard-Jones potential
\[
\varphi_{i,i+r}(z) = -2 \left(\frac{z}{l_{i,i+r}}\right)^{-6} + \left(\frac{z}{l_{i,i+r}}\right)^{-12}
\quad (1\le r\le R)
\]
with the varying equilibrium distance
\[
l_{i,i+r} = \left\{
	\begin{array}{lcl}
	1 & & \textnormal{$i$ is even} \\
	9/8 & & \textnormal{$i$ is odd.}
	\end{array}
\right.
\]
The external force was taken as
\[
f_i = 50 \sin\left(1+2\pi X_i\right).
\]

\begin{figure}
\begin{center}
\hfill
	\includegraphics{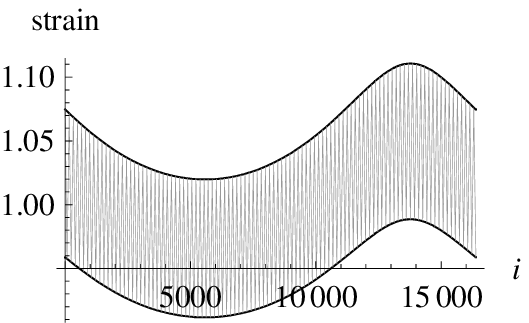}
\hfill\hfill
	\includegraphics{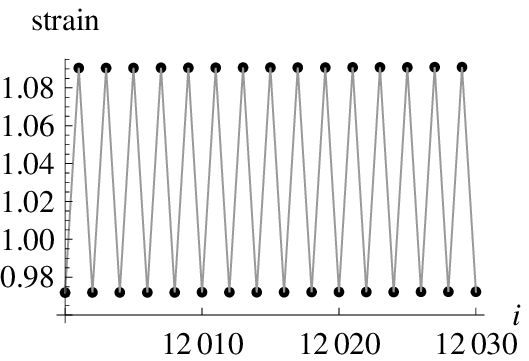}
\hfill $\mathstrut$
\end{center}
\caption{Strain $D u_i$ of the solution of the 1D nonlinear problem: the schematically shown complete solution (left) and the closeup of the micro-structure for 31 atoms (right).}
\label{fig:solution-nonlinear}
\end{figure}

The strain $D u_i$ for such problem is shown in Fig.\ \ref{fig:solution-nonlinear}.

\begin{figure}
\begin{center}
\includegraphics{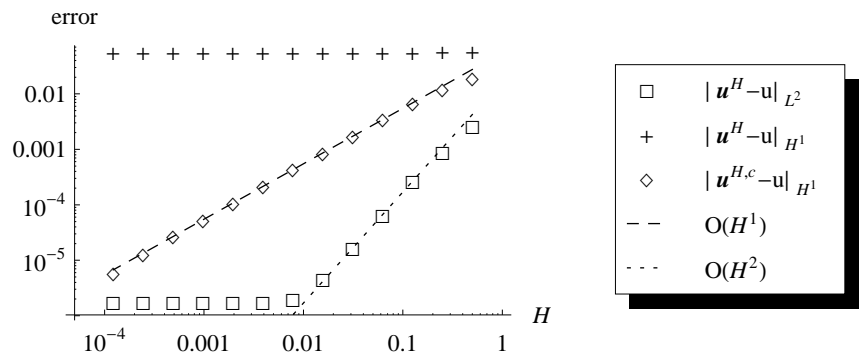}
\end{center}
\caption{Results for the 1D nonlinear problem: errors of the post-processed HQC solution $u^{H,c}$ and the homogenized solution $u^H$ in different norms.
The errors behave in accordance with Theorems \ref{thm:analysis-of-methods:homogenized-L2-convergence} and \ref{thm:analysis-of-methods:reconstructed-convergence}.}
\label{fig:error1-nonlinear}
\end{figure}

Figure \ref{fig:error1-nonlinear} plots the errors of HQC in $H^1$- and $L^2$-norms.
The results observed are qualitatively the same as the results on Fig.\ \ref{fig:error2-linear} obtained for the linear case.
Also, the results of a naive application QC method to the nonlinear problem will be the same as shown on Fig.\ \ref{fig:error2-linear} for the linear case.
Thus, we conclude that the convergence estimates derived for the linear case are also valid for the nonlinear case, as the conducted numerical experiments show.

\subsection{Convergence for Different Periods $p$}\label{sec:numeric:different-p}

\begin{table}
\begin{center}
\begin{tabular}{|c|cccc|}
\hline
& $p=2$ & $p=4$ & $p=8$ & $p=16$ \\ \hline
1st test case (linear) & $0.040$ & $0.043$ & $0.044$ & $0.042$ \\
2nd test case (nonlinear) & $0.019$ & $0.018$ & $0.015$ & $0.017$ \\
\hline
\end{tabular}
\end{center}
	\caption{Dependence of the bound $C_8=\max\limits_H \frac{|u^{H,c}-u|_{H^1}}{H}$ (cf.\ Theorem \ref{thm:analysis-of-methods:reconstructed-convergence}) on $p$.
	It can be seen that $C_8$ essentially does not depend on $p$.}
	\label{tab:different-p}
\end{table}

In our analysis of the equations and the computational method, we kept the dependence on $p$ implicit, because derivation of estimates which are sharp w.r.t.\ $p$ is much more technical.
In this section we numerically address this issue.

The first test case is similar to the one in Section \ref{sec:numeric:1d-linear}.
We fixed the bounds for $k_{i,i+r}$ in \eqref{eq:numeric:1d-linear:k} between $1$ and $2$ and randomly generated the values of $k_{i,i+r}$ with the periods $p=2,4,8,16$.
We estimate the constant $C_8$ in Theorem \ref{thm:analysis-of-methods:reconstructed-convergence} as
\[
C_8 = \max\limits_{H} \frac{|u^{H,c}-u|_{H^1}}{H},
\]
where the maximum is taken for $H=2^{-1}, 2^{-2}, \ldots, 2^{-7}$.
The second test case is similar to the one in Section \ref{sec:numeric:1d-nonlinear} with $l_{i,i+r}$ randomly generated between $1$ and $\frac{11}{10}$.

The results for both test cases are shown in Table \ref{tab:different-p}.
It can be seen that the constant $C_8$ essentially does not depend on $p$.

This finding is important in applications, for instance, to shape memory alloys that may change their crystalline structure in the course of loading/unloading.
Motivated by such applications, the authors of \cite{DobsonElliottLuskinEtAl2007} designed the adaptive strategy of choosing $p$, called Cascading Cauchy-Born kinematics, for the complex lattice QC method.
They also presented an example of application of their method to the 1D model problem exhibiting period-doubling bifurcations.
The present findings indicate that increase of period $p$ does not affect the accuracy of the method.

Independence of error bounds on $p$ is also important for modeling amorphous materials, such as glasses or polymers.
Amorphous materials do not have a spatial period, instead they exhibit some random structure.
By analogy with application of numerical homogenization to PDEs with random tensors, one could take $p$ large enough to capture the variation of the microscopic structure of the amorphous material, and expect that it will not affect the accuracy of representation of the macroscopic deformation as the mesh size $H$ is refined.

\subsection{2D Test Case 1}\label{sec:numeric:2d-first}

\begin{figure}
\begin{center}
\hfill
\includegraphics{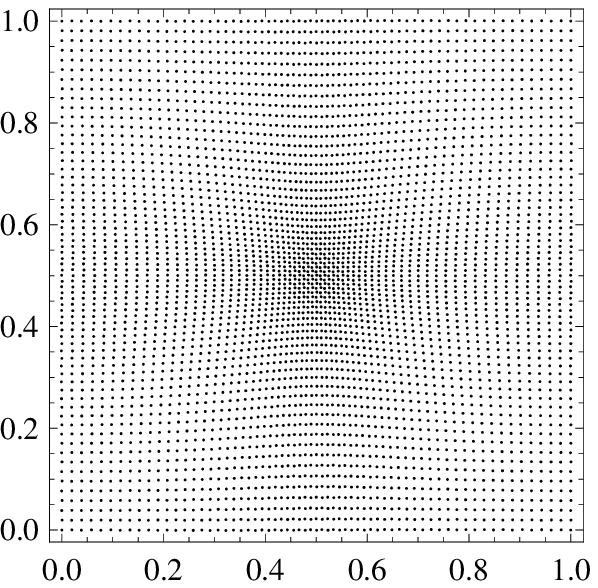}
\hfill
\includegraphics{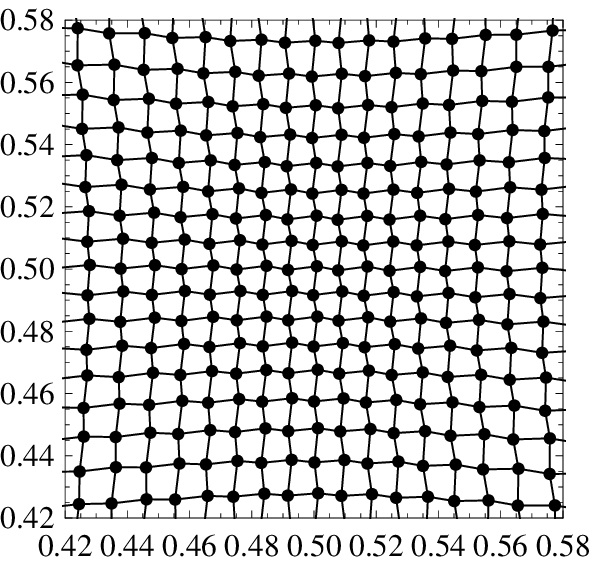}
\hfill
\end{center}
\caption{Atomic equilibrium configuration for $N_1=N_2=64$ for the 2D test case 1.
	Deformation of the whole material (left) and a close-up (right).}
\label{fig:2d-solution}
\end{figure}

We consider the example of material discussed in Subsection \ref{sec:2d-example}, with $\epsilon=2^{-11}$, $N_1=N_2=2^{11}$, $k_1=1$, $k_2=2$, $k_3=0.25$,
\[
f_i = 10 e^{-\cos(\pi i_1/N_1)^2-\cos(\pi i_2/N_2)^2} \left(\begin{array}{c} \sin(2\pi i_1/N_1) \\ \sin(2\pi i_2/N_2) \end{array}\right)
-\bar{f},
\]
where $\bar{f}$ is determined so that the average of $f_i$ is zero.
The total number of degrees of freedom of such system is approximately $8\cdot 10^{6}$.
The solution for such test case is shown in fig.\ \ref{fig:2d-solution} (the illustration is for $N_1=N_2=64$).

\begin{figure}
\begin{center}
\includegraphics{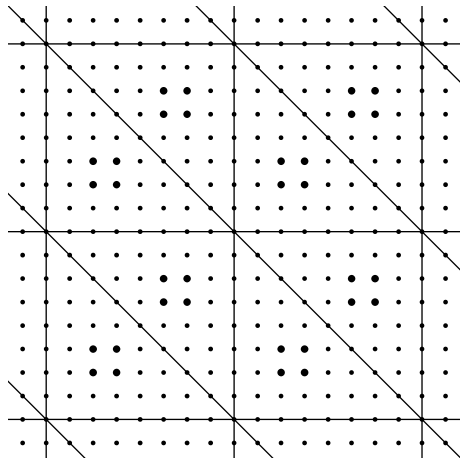}
\end{center}
\caption{Illustration of a 2D triangulation.
	Larger atoms comprise sampling domains for HQC.}
\label{fig:2d-triangulation}
\end{figure}

The atomistic domain is triangulated using $t^2$ nodes and $K = 2 t^2$ triangles ($t=2,4,\ldots,2^{10}$).
In each triangle $S_k$ a sampling domain ${\mathcal I}_k$ is chosen, each sampling domain contains four atoms (see illustration in fig.\ \ref{fig:2d-triangulation}).
The number of degrees of freedom of the discretized problem is $2 t^2$.

\begin{figure}
\begin{center}
\includegraphics{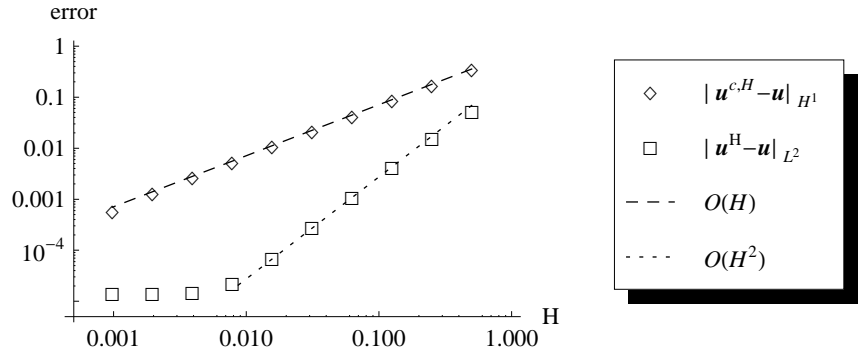}
\end{center}
\caption{Results for the 2D test case 1: error depending on the mesh size $H$.
The $L^2$-error of the homogenized solution $u^H$ and the $H^1$-error of the post-processed solution $u^{H,c}$ are shown.
The errors behave in accordance with the 1D analysis (Theorems \ref{thm:analysis-of-methods:homogenized-L2-convergence} and \ref{thm:analysis-of-methods:reconstructed-convergence}).}
\label{fig:2d-testcase1-error}
\end{figure}

The error of the solution for different mesh size $H$ ($H=0.5,0.25,\ldots,2^{-10}$) is shown in fig.\ \ref{fig:2d-testcase1-error}.
The results are essentially the same as in 1D case: the method convergences with the first order of mesh size in the $H^1$-norm and with the second order in the $L^2$-norm.
We also see the plateau for the $L^2$-error of the homogenized solution $u^{c, H}$.
It is remarkable that all the conclusions of 1D analysis (cf.\ Theorems \ref{thm:analysis-of-methods:homogenized-L2-convergence} and \ref{thm:analysis-of-methods:reconstructed-convergence}) are also valid for the 2D computations.

\subsection{2D Test Case 2}\label{sec:numeric:2d-second}

\begin{figure}
\begin{center}
\hfill
\includegraphics{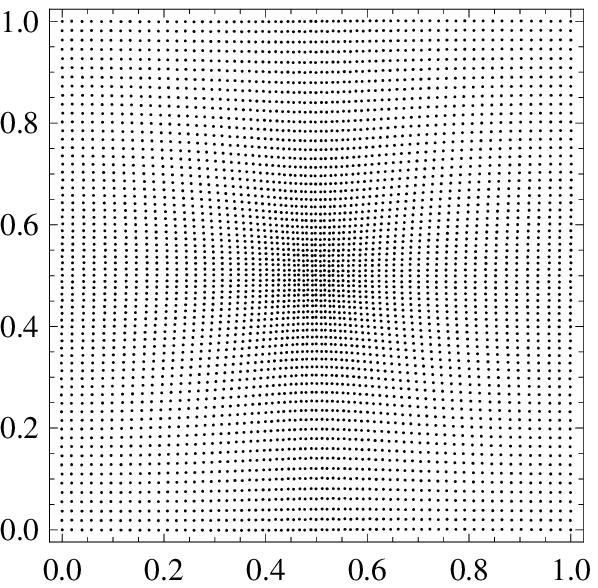}
\hfill
\includegraphics{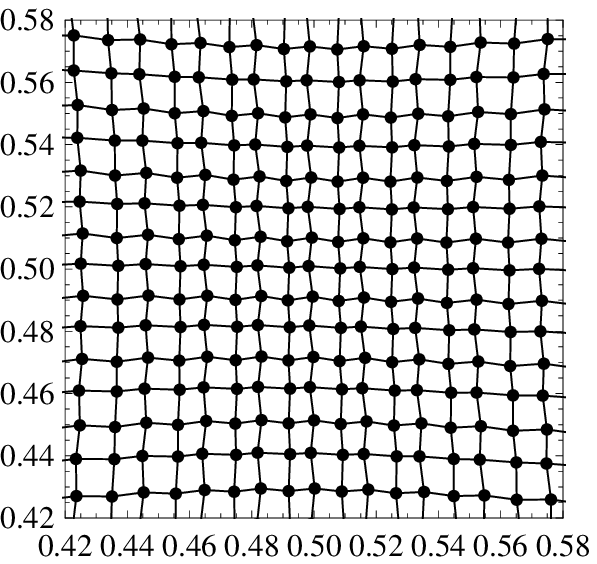}
\hfill
\end{center}
\caption{Atomic equilibrium configuration for $N_1=N_2=64$ for the 2D test case 2.
Deformation of the whole material (left) and a close-up (right).}
\label{fig:2d-solution-rand}
\end{figure}

The second test case is analogous to the previous one, but with the different tensors $\psi_r$ describing the atomistic bonds.
The tensors $\psi_r$ were chosen to have the following (randomly generated) values:
\[
\psi_{(1,0),i} = \left\{
	\begin{array}{lcl}
	1.3 & & \textnormal{$i_1$ even, $i_2$ even} \\
	1.6 & & \textnormal{$i_1$ even, $i_2$ odd} \\
	1.8 & & \textnormal{$i_1$ odd, $i_2$ even} \\
	1.2 & & \textnormal{$i_1$ odd, $i_2$ odd,}
	\end{array}
\right.
\quad
\psi_{(0,1),i} = \left\{
	\begin{array}{lcl}
	1.5 & & \textnormal{$i_1$ even, $i_2$ even} \\
	1.7 & & \textnormal{$i_1$ even, $i_2$ odd} \\
	1.5 & & \textnormal{$i_1$ odd, $i_2$ even} \\
	2 & & \textnormal{$i_1$ odd, $i_2$ odd,}
	\end{array}
\right.
\]
\[
\psi_{(1,1),i} = \left\{
	\begin{array}{lcl}
	0.3 & & \textnormal{$i_1$ even, $i_2$ even} \\
	0.8 & & \textnormal{$i_1$ even, $i_2$ odd} \\
	0.6 & & \textnormal{$i_1$ odd, $i_2$ even} \\
	0.4 & & \textnormal{$i_1$ odd, $i_2$ odd,}
	\end{array}
\right.
\quad
\psi_{(-1,1),i} = \left\{
	\begin{array}{lcl}
	0.4 & & \textnormal{$i_1$ even, $i_2$ even} \\
	0.9 & & \textnormal{$i_1$ even, $i_2$ odd} \\
	0.4 & & \textnormal{$i_1$ odd, $i_2$ even} \\
	0.1 & & \textnormal{$i_1$ odd, $i_2$ odd.}
	\end{array}
\right.
\]
For a tensor with such a random structure, the homogenized tensor can only be precomputed numerically, and in the case of a nonlinear problem should be found in the course of the actual computation.
The solution (for $N_1=N_2=64$) is shown in fig.\ \ref{fig:2d-solution-rand}.

\begin{figure}
\begin{center}
\includegraphics{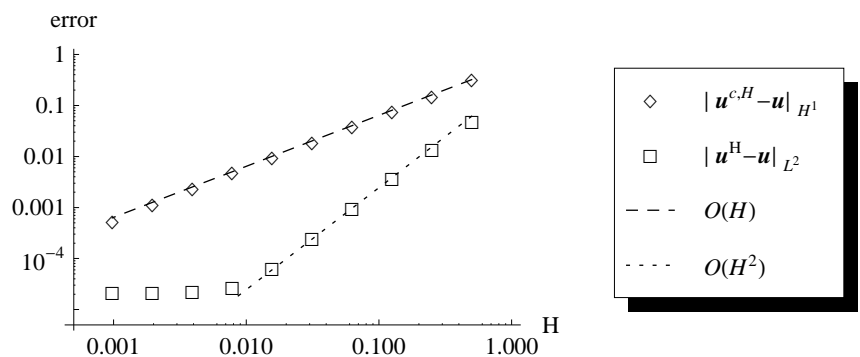}
\end{center}
\caption{Results for the 2D test case 2: error depending on the mesh size $H$.
The $L^2$-error of the homogenized solution $u^H$ and the $H^1$-error of the post-processed solution $u^{H,c}$ are shown.
The errors behave in accordance with the 1D analysis (Theorems \ref{thm:analysis-of-methods:homogenized-L2-convergence} and \ref{thm:analysis-of-methods:reconstructed-convergence}).}
\label{fig:2d-testcase2-error}
\end{figure}

The error of the solution for different number of degrees of freedom is shown in fig.\ \ref{fig:2d-testcase2-error}.
The results are similar to the results of all the previous test problems.
Again, the results are in accordance with Theorems \ref{thm:analysis-of-methods:homogenized-L2-convergence} and \ref{thm:analysis-of-methods:reconstructed-convergence}.

\section{Summary and Conclusion}\label{sec:conclusion}

We have considered the problem of modeling materials with complex atomistic lattice.
We have proposed a discrete homogenization framework to analyze the QC method for complex crystalline materials.
This framework allowed us to prove convergence (in 1D) for the QC method proposed in 
\cite{TadmorSmithBernsteinEtAl1999}. Numerical homogenization has also been used to formulate
the QC method.  
The equivalence of this algorithm to the QC method of \cite{TadmorSmithBernsteinEtAl1999}  is discussed in detail in \cite{AbdulleLinShapeevII}.
We have also shown how to apply the presented technique in a 2D setting.
The 1D and 2D numerical examples presented verify validity of the analysis in more general setting.
We note that the extension of the algorithm proposed in this paper to simulate atomistic materials at finite temperature or non-crystalline materials is of high interest.
This is a topic for future research.

\bibliographystyle{siam}
\bibliography{hqc}

\section*{Appendix A}
\begin{lemma}[Discrete Poincar\'e inequality 1]\label{lem:discrete-poincare}
Let $\vec{g} \in \bbR^L$ and $\sum_{i=1}^L g_i = 0$.
Then
\begin{equation}
\sum\limits_{i=1}^L |g_i|^2 \le \frac{L^2}{6}\sum\limits_{i=1}^{L-1} |g_{i+1}-g_i|^2.
\label{eq:discrete-poincare}
\end{equation}
\end{lemma}
\begin{proof}
We start with noticing that lemma A.1 in \cite[p.\ 87]{OrtnerSuli2008} applies to $\vec{g}$ and states that
\[
|g_i| \le \sum\limits_{j=1}^{L-1} |g_{j+1}-g_j| \phi_{i,j},
\]
where
\[
\phi_{i,j} = \left\{
	\begin{array}{lcl}
	\frac{j}{L} & & j\le i \\
	\frac{L-j}{L} & & j> i.
	\end{array}
\right.
\]
Then with the help of the Cauchy-Schwarz inequality one obtains
\[
\sum\limits_{i=1}^L |g_i|^2
\le
\sum\limits_{i=1}^L \left(\sum\limits_{j=1}^{L-1} |g_{j+1}-g_j|\phi_{i,j}\right)^2
\le
\sum\limits_{i=1}^L \left(\sum\limits_{j=1}^{L-1} |g_{j+1}-g_j|^2\right) \left(\sum\limits_{j=1}^{L-1} \phi_{i,j}^2\right),
\]
where by direct computation
\begin{align*}
\sum\limits_{i=1}^L \sum\limits_{j=1}^{L-1} \phi_{i,j}^2
=~&
\sum\limits_{j=1}^{L-1} \sum\limits_{i=1}^L \phi_{i,j}^2
=
\sum\limits_{j=1}^{L-1} \left(\left(\frac{L-j}{L}\right)^2 j + \left(\frac{j}{L}\right)^2 (L-j)\right)
\\ =~&
\sum\limits_{j=1}^{L-1} \frac{(L-j) j}{L} = \frac{L^2-1}{6} < \frac{L^2}{6}.
\end{align*}
\end{proof}

\begin{lemma}\label{lem:representation}
Let $\vec{g} \in \left(\bbR^L\right)_{\#}$.
Then
\[
g_i = \sum\limits_{k=1}^L \frac{L+1-2 k}{2 L}\,(g_{i-k+1}-g_{i-k}).
\]
\end{lemma}
\begin{proof}
Direct computation of the right-hand-side (RHS) yields:
\begin{align*}
{\rm RHS}
=~&
\sum\limits_{k=1}^L \frac{L+1-2 k}{2 L}\,\left(g_{i-k+1}-g_{i-k}\right)
\\ =~&
\sum\limits_{k=0}^{L-1} \frac{L+1-2 (k-1)}{2 L}\,g_{i-k}
-
\sum\limits_{k=1}^L \frac{L+1-2 k}{2 L}\,g_{i-k}
\\ =~&
\frac{L-1}{2 L} g_i
- \sum\limits_{k=1}^{L-1} \frac{2}{2 L}\,g_{i-k}
+ \frac{L-1}{2 L} g_{i-L}.
\end{align*}
Notice that due to periodicity $g_{i-L}=g_i$ and due to average of $\vec{g}$ being zero, $\sum\limits_{k=1}^{L-1} g_{i-k} = -g_i$.
Hence
\[
{\rm RHS} = \frac{L-1}{2 L} g_i + \frac{2}{2 L} g_i + \frac{L-1}{2 L} g_i = g_i.
\]
\end{proof}

If we consider the periodic extension of the sequence then the estimate of lemma \ref{lem:discrete-poincare} will have a slightly better constant:
\begin{lemma}[Discrete Poincar\'e inequality 2]
Let $\vec{g}\in \left(\bbR^L\right)_{\#}$.
Then
\[
\sum\limits_{i=1}^L |g_i|^2 \le \frac{L^2}{12} \sum\limits_{i=1}^{L} |g_{i+1}-g_i|^2.
\]
\end{lemma}
\begin{proof}
Then by using lemma \ref{lem:representation} and the Cauchy-Schwarz inequality one obtains
\begin{align*}
\sum\limits_{i=1}^L |g_i|^2
\le~&
\sum\limits_{i=1}^L \left(\sum\limits_{k=1}^{L} \frac{L+1-2 k}{2 L} (g_{i+1-k}-g_{i-k})\right)^2
\\ \le~&
\sum\limits_{i=1}^L \sum\limits_{k=1}^{L} \left(\frac{L+1-2 k}{2 L}\right)^2 \sum\limits_{k=1}^{L} (g_{i+1-k}-g_{i-k})^2
\\ =~&
\sum\limits_{i=1}^L \frac{L^2-1}{12 L} \sum\limits_{k=1}^{L} (g_{i+1-k}-g_{i-k})^2
\\ =~&
\frac{L^2-1}{12} \sum\limits_{j=1}^{L} (g_{j+1}-g_{j})^2
\le
\frac{L^2}{12} \sum\limits_{j=1}^{L} (g_{j+1}-g_{j})^2.
\end{align*}
\end{proof}

\begin{corollary}[Discrete Poincar\'e inequality for $U_{\#}^n$]
\label{cor:discrete-poincare-for-Unsharp}
The functional $|\bullet|_{H^1}$ (cf.\ \eqref{eq:estimates:norms_fct}) defines a norm on $U_{\#}^n$.
For $\vecu \in U_{\#}^n$ the following inequality holds:
\[
\|u\|_{L^2} \le \frac{1}{2 \sqrt{3}} |u|_{H^1}.
\]
\end{corollary}

\begin{lemma}[Inverse discrete Poincar\'e inequality]\label{lem:inverse-Poincare}
For $\vecu \in U_{\rm rep}^n$ the following inequality holds:
\begin{equation}
\epsilon |\vecu |_{W^{1,q}} \le 2 \|\vecu \|_{L^q}
\label{eq:inverse-Poincare}
\end{equation}
for $1\le q \le \infty$.
\end{lemma}
\begin{proof}
\[
\epsilon |\vecu |_{W^{1,q}}
=
\|\epsilon D \vecu \|_{L^q}
=
\|T_i \vecu - \vecu \|_{L^q}
\le
\|T_i \vecu \|_{L^q} + \|\vecu \|_{L^q}
= 2 \|\vecu \|_{L^q}.
\]
\end{proof}

\begin{lemma}\label{lem:Hminus1-through-L2}
For $\vecu \in U_{\#}^n$ the following inequality holds:
\begin{equation}
|\vecu |_{H^{-1}} \le \frac{1}{2 \sqrt{3}} \|\vecu \|_{L^2}.
\label{eq:Hminus1-through-L2}
\end{equation}
\end{lemma}
\begin{proof}
Using the discrete Poincar\'e inequality yields:
\[
|\vecu |_{H^{-1}}
=
\sup\limits_{\substack{\vec{v}\in U_{\#}^n \\ \vec{v}\ne 0}} \frac{\left<\vecu ,\vec{v}\right>_i}{|\vec{v}|_{H^1}}
\le
\sup\limits_{\substack{\vec{v}\in U_{\#}^n \\ \vec{v}\ne 0}} \frac{\left<\vecu ,\vec{v}\right>_i}{2 \sqrt{3} \|\vec{v}\|_{L^2}}
=
\frac{1}{2 \sqrt{3}}\|\vecu \|_{L^2}.
\]
\end{proof}

\end{document}